\documentclass[12pt]{amsart}

\usepackage{amsmath}
\usepackage{amsfonts}
\makeatletter
\def\amsbb{\use@mathgroup \M@U \symAMSb}
\makeatother

\usepackage{cancel}
\usepackage{bbold}

\usepackage{adjustbox}
\usepackage{graphicx}
\usepackage{verbatim}
\usepackage{textcomp}
\usepackage{amssymb}
\usepackage{cite}
\usepackage{color}
\usepackage[all]{xy}
\usepackage{graphicx}
\usepackage{color}
\usepackage{tikz-cd}
\usepackage{hyperref}
\usepackage{soul}
\hypersetup{
    colorlinks,
    citecolor=black,
    filecolor=black,
    linkcolor=black,
    urlcolor=black
}

\usepackage{mathtools}
\usepackage[normalem]{ulem}
\usepackage{cancel}
\usepackage{tikz}
\usepackage[all]{xy}
\usepackage{graphicx}
\usetikzlibrary{backgrounds}
\newtheorem{theorem}{Theorem}[section]
\newtheorem{lemma}[theorem]{Lemma}

\allowdisplaybreaks

\usepackage{tikz}
\usetikzlibrary{arrows.meta, decorations.markings}

\theoremstyle{definition}
\newtheorem{definition}[theorem]{Definition}
\newtheorem{remark}[theorem]{Remark}

{{\sc Proof of {Theorem~\ref{RegularityTheorem} }and Theorem~\ref{goto0}.}}
{{\sc q.e.d.} \\}

{{\sc Proof of Theorem~\ref{rigiditytheorem}.}}%
{{\sc q.e.d.} \\}
{{\sc Proof of Lemma~\ref{uscord}.}}%
{{\sc q.e.d.} \\}
{{\sc Proof of Lemma~\ref{federeroutermeasure}.}}%
{{\sc q.e.d.} \\}
{{\sc Proof of Lemma~\ref{cd2general}.}}%
{{\sc q.e.d.} \\}

{{\sc Proof of Lemma~\ref{usordv}.}}%
{{\sc q.e.d.} \\}
{{\sc Proof of Corollary~\ref{symTeich}.}}%
{{\sc q.e.d.} \\}
{{\sc Proof of Theorem~\ref{surf}.}}%
{{\sc q.e.d.} \\}
{{\sc Proof of Theorem~\ref{hmco2}.}}%
{{\sc q.e.d.} \\}
{{\sc Proof of claim~(\ref{rickclaim}).}}%
{{\sc q.e.d.} \\}

{{\sc Proof of Theorem~(\ref{RegularityTheoremModelSpace}).}}%
{{\sc q.e.d.} \\}
\numberwithin{equation}{section}

\newtheorem*{theorem*}{Theorem}
{{\sc Proof of Lemma~\ref{tri1}.}}%
{{\qed} \\}
{{\sc Proof of Theorem~\ref{regularity}.}}%
{{\qed} \\}
{{\sc Proof of Theorem~\ref{main}.}}%
{{\qed} \\}
\newenvironment{proofof(i)}%
    {{\sc Proof of $(i)$.}}%
  {{\qed} \\}  
  
  \newenvironment{proofof(iv)}%
    {{\sc Proof of $(iv)$.}}%
  {{\qed} \\}  

\newcommand{\ms}{{\bf H}}
\newcommand{\bms}{\overline{\bf H}}

\newcommand{\R}{\mathbb R}

\newcommand{\N}{\mathbb N}
\newcommand{\C}{\mathbb C}

\title{Regularity of Harmonic Maps into Teichmüller Space}
\author{Yitong Sun}
\begin{document}
\maketitle

\begin{abstract}
We prove a regularity theorem for harmonic maps into Teichmüller space. More specifically, if $u$ is a harmonic map from a Riemannian domain to the metric completion of Teichmüller space with respect to the Weil-Petersson metric, and the image of $u$ intersects a stratum of the augmented Teichmüller space, then $u$ is entirely contained in this stratum. This extends Wolpert's result on the geodesic convexity of the augmented Teichmüller space to higher dimensions and generalizes the regularity result of Daskalopoulos and Mese by showing that the singular set of $u$ is empty. 
\end{abstract}

\section{Introduction}

Teichm\"uller space has been a subject of intense interest to many mathematicians since its introduction in the 1940s. Complex analytic foundations were laid by L.~Ahlfors, L.~Bers, H.~Royden, and S.~Earle, among others. Later, W.~Thurston revolutionized the field by connecting Teichm\"uller space with hyperbolic geometry. The introduction of the Weil--Petersson metric endowed Teichm\"uller space with rich geometric and analytic structures. The Weil--Petersson metric on Teichmüller space provides a deep connection between the complex analytic structure of moduli spaces and the hyperbolic geometry of surfaces. They were extensively studied by A.~Weil, S.~Wolpert, H.~Masur, W.~Harvey, Y.~Minsky, C.~McMullen, J.~Brock, F.~Gardiner, L.~Keen, and many others. The use of harmonic map theory to study its global structure has led to deep results in compactification theory and rigidity, as seen in the works of M.~Wolf, Y.~Minsky, S.~Yamada, R.~Wentworth, G.~Daskalopoulos, C.~Mese, and others.

We focus on harmonic maps into the Weil--Petersson metric completion of Teichm\"uller space $\overline{\mathcal{T}}$ (cf. Daskalopoulos -- Mese \cite{1}). Wolpert \cite{5} showed that Teichmüller space endowed with Weil--Petersson metric completion is geodesically convex. Teichmüller space $\mathcal{T},$ which parametrizes complex structures on an oriented surface of genus $g$ with $p$ marked points, becomes an incomplete and non-positive curvature (NPC) space when equipped with the Weil--Petersson (WP) metric. Its metric completion $(\overline{\mathcal T},d_{wp})$ -- which includes nodal surfaces where simple closed curves are pinched -- is an NPC metric space satisfying CAT(0) property. The augmented Teichmüller space $\overline{\mathcal T}$ is a stratified space where each lower-dimensional open stratum $\mathcal{T}'$ parametrizes surfaces derived from the original oriented surface with a number of nodes. Each $\mathcal{T}'$ is a product of lower-dimensional Teichmüller spaces. Given two points in a stratum $\mathcal T'$, the geodesic connecting them is contained in $\mathcal T'$. In other words, if a geodesic curve $\gamma$ intersects a stratum $\mathcal T'$ of $\overline{\mathcal T}$, then $\gamma \subset \mathcal T'$. 

We generalize Wolpert's result by proving the same result for harmonic maps. Specifically, we establish that if a harmonic map from a Riemannian domain into $\overline{\mathcal{T}}$ intersects a lower-dimensional open stratum, then its entire image lies within that stratum. The main result of this paper is the following statement.
\begin{theorem}\label{mainresult}
Let $\Omega$ be a Riemannian domain, $\overline{\mathcal{T}}$ be the metric completion of $\mathcal{T}$ with respect to the Weil--Petersson metric, $u : \Omega \to \overline{\mathcal{T}}$ be a harmonic map, and $\mathcal{T}'$ be a stratum of $\overline{\mathcal{T}}$. If $u(\Omega) \cap \mathcal{T}' \neq \emptyset$, then $u(\Omega) \subset \mathcal{T}'$.
\end{theorem}

A key step of the proof of Theorem \ref{mainresult} is the following theorem:
\begin{theorem}\label{nosing}
    Let $u: \Omega \to \overline{\mathcal{T}}$ be a harmonic map from a Riemannian domain to the metric completion of $\mathcal{T}$ endowed with Weil--Petersson metric. If $u$ intersects $\mathcal{T}$ at some point, then $u$ has no singular points and is, in fact, smooth harmonic map into $\mathcal{T}.$
\end{theorem}
Theorem \ref{nosing} completes the circle of ideas initiated by Daskalopoulos and Mese by simplifying the original argument so that it applies to the higher order points. In particular, to prove the holomorphic rigidity of Teichm\"uller space, Daskalopoulos and Mese \cite{1} showed that harmonic maps into $\overline{\mathcal{T}}$ are sufficiently regular to permit the application of Siu’s Bochner technique. They proved that a harmonic map from an $n$-dimensional smooth Riemannain domain to $\overline{\mathcal{T}}$ doesn't have order 1 {\it singular points}, which are points mapped to the boundary of $\mathcal{T},$ and its {\it singular set} has dimension $\leq n-2.$ Theorem \ref{nosing}, establishing a regularity theory, paves the way for applying harmonic techniques. 



\subsection{Outline of this paper}
In section 2, we briefly describe the basic concepts related to this paper. We define the {\it model space} $\bf H$ and its metric completion $\overline{\bf H} := {\bf H} \cup \{P_0\}$ in section \ref{sec:modelspace}, where we identify the boundary of $\bf H$ by the single $P_0.$ We also introduce {\it symmetric geodesics} in $\bf H$ and {\it a metric space $\overline{\bf H}_A$} in sections \ref{2.2} and \ref{2.3}. Symmetric geodesic is an important tool in the following sections to approximate the image of a pullback limit of the sequence of blow-up maps (cf. Lemma \ref{linearapproxblowupsym0}). We explain the local coordinates near the boundary of the augmented Teichmüller space $\overline{\mathcal{T}}$ concisely in the end of the section. Augmented Teichmüller space $\overline{\mathcal{T}}$ of dimension $k$ is a stratified space and each boundary point is contained in a $j$-dimensional stratum for some $j <k.$ The neighborhood near the boundary point $P \in \partial \mathcal{T}$ is asymptotically isometric to a product space of a $j$-dimensional smooth open stratum $\mathcal{T}'$ and $\overline{\bf H}^{k-j}=\overline{\bf H} \times ... \times \overline{\bf H}.$

Section 3 focuses on harmonic maps into the model space. In this section, we prove that non-constant harmonic maps into $\overline{\bf H}$ are smooth, i.e. avoid the boundary (cf. Theorem \ref{result0}). Since $\overline{\bf H}$ captures singular features near $\partial \mathcal{T},$ all the key ideas for the main theorem appear in this section. To prove Theorem \ref{result0}, we construct a tangent map for our harmonic map $u:\Omega \to \overline{\bf H}$ and use its structure to get the result.
In particular, applying the modification factor $\lambda^u,$ we construct a sequence $\{u_k\}$ of non-constant \emph{harmonic blow-up maps} converging locally uniformly to a tangent map $u_*$ in a pullback sense. In this setting, $u_*$ is a homogeneous harmonic map into the metric space $\overline{\bf H}_A.$ The structure of $u_*$ implies necessary distance estimates (cf. Lemmas \ref{fits0} -- \ref{s30}), which are the key step in the proof of Theorem \ref{result0}.

In section \ref{sec:blowup}, we aim to prove Theorem \ref{nosing}. Under the local coordinates of $\overline{\mathcal{T}}$ near $\partial \mathcal{T},$ we assume on the contrary that the singular set of $u$ is non-empty and pick a singular point $x_0$ such that $u(x_0) \in \partial \mathcal{T}$ is contained in a $j$-dimensional stratum $\mathcal{T}'$. Analogously to \cite{1}, we decompose the harmonic map $u$ near the singular point $x_0$ as $u=(V,v)$ where $V$ is called {\it regular component} mapping to $\mathcal{T}'$ and $v$ is called {\it singular component} mapping to $\overline{\bf H}^{k-j}.$ Following from the hypothesis that $u(\Omega) \cap \mathcal{T} \neq \emptyset,$ the singular map $v:B_{r_0}(x_0) \to \overline{\bf H}^{k-j}$ has all non-constant component maps $v^{\eta}:B_{r_0}(x_0) \to \overline{\bf H}$ with $v^{\eta}(x_0)=P_0.$ We construct the sequence $\{v_{\sigma_i}: B_1(0) \to \overline{\bf H}^{k-j}\}$ of \emph{blow-up maps} and show a subsequence of $\{v_{\sigma_i}\}$ converges to a homogeneous harmonic limit map $v_*$ in the pullback sense. We have two cases: (i) there exists a non-constant component $v^{\eta_0}_*$ of pullback limit $v_*$ or (ii) $v_*$ is constant.

Section \ref{sec:blowup1} shows that $v^{\eta_0}(x_0) \neq P_0,$ which implies that $x_0$ is not a singular point of $u$ and then the singular set of $u$ is empty. However, unlike $u$ in section \ref{sec:blowup0}, $v$ is not a harmonic map because WP-metric is {\it only} asymptotically the product metric. To resolve this difficulty, we construct a sequence of {\it approximating harmonic maps}, which is the essential tool in replacing the non-harmonic map $v^{\eta_0}_{\sigma_i}$ by the harmonic map $w^{\eta_0}_i$ in the subsequent arguments (cf. Lemma \ref{re}). The proof of Theorem \ref{nosing} then proceeds analogously to the method in section \ref{sec:blowup0}. Section \ref{sec:blowup2} is for the case that all components $v^{\eta}_*$ are constant. To handle this complexity, we introduce a modified scaling factor $\lambda^v,$ replacing the earlier factor of $\lambda^u,$ to construct the new sequence $\{\tilde{v}_{\sigma_i}:B_1(0) \to \overline{\bf H}^{k-j}\}$ of \emph{alternative blow-up maps}. Then, the idea follows the steps in section \ref{sec:blowup1} with some further adjustments due to the changing of the factor $\lambda^v$.

Section \ref{pfmain} constitutes the proof of Theorem \ref{mainresult}. We prove Theorem \ref{mainresult} by invoking the results from the analysis in previous sections with the assumption $u(\Omega) \cap \mathcal{T}\neq \emptyset$ replaced by $u(\Omega) \cap \mathcal{T}' \neq \emptyset$.

\subsection{Main Concepts} Let $u: \Omega \to (\overline{\mathcal{T}},d_{wp}).$ We recall these fundamental ideas and provide references:
\begin{itemize}
    \item {\it order of a harmonic map} (cf. \cite[Section 2]{4}): The order of a harmonic function is the degree of the dominant term in the homogeneous harmonic polynomial approximating $u-u(x)$ near $x$.
    \item {\it blow-up maps $u_{\sigma}$ at $x_0$} (cf. \cite[Section 3]{4}): Using normal coordinates centered at $x_0,$ we identify $x_0=0.$ We restrict $u$ to a ball $B_{\sigma}(x_0) \subset \Omega$ where factor $\sigma>0$ is close to zero, and rescale the domain of $u$ by the factor $\sigma$ and the distance by the factor $\lambda^u(\sigma),$ where $\lambda^u(\sigma)$ is approaching to infinity as $\sigma \to 0.$ The scaling map is called the blow-up map $u_{\sigma}: B_1(0) \to (\overline{\mathcal{T}},\lambda^u(\sigma)d_{wp})$ where $u_{\sigma}(x)=\lambda^u(\sigma) u(\sigma x).$
    
\end{itemize}

\section{Preliminaries}
\subsection{Model Space} \label{sec:modelspace}
Model space is a crucial tool when studying $\overline{\mathcal{T}}$ because it provides a lower-dimensional, explicitly defined setting for the boundary geometry of $\overline{\mathcal{T}}.$ Near the boundary, $\overline{\mathcal{T}}$ is asymptotically isometric to the product of a smooth open stratum $\mathcal{T}'$ with the structure of a Kähler manifold and a metric space $\overline{\bf H}$ or $\overline{\bf H} \times ... \times \overline{\bf H}$ (cf. \cite{1} and \cite{6}). The model space $\overline{\bf H}$ captures key singular features of $\overline{\mathcal{T}},$ such as the sectional curvature blow-up near $\partial \mathcal{T}$ and the non-local compactness of $\overline{\mathcal{T}},$ which are also properties of $\overline{\bf H}.$

Let $({\ms}, g_{\ms})$ be the {\it model space} of \cite[Section 2.1]{1}; i.e.
 \[
 {\bf H}=\{(\rho,\phi) \in \R^2: \rho >0, \phi \in \R\}\]
and
\[
g_{\bf H}=d\rho^2+ \rho^6 d\phi^2.
\]
We will call $(\rho,\phi)$ the {\it standard model space coordinates}. By direct computation, we obtain that ${\bf H}$ has negative Gauss curvature. 
The distance function defined by $g_{\bf H}$ will be denoted as $d_{\bf H}$. 
Let $(\bms, d_{\overline{\ms}})$ where $\bms:=\ms \cup \{P_0\}$ be the metric completion of the metric space $({\ms}, d_{\ms})$. Note that $(\overline{\bf H},d_{\overline{\bf H}})$ is an NPC space because it's a metric completion of the geodesically convex surface ${\bf H}$ with negative curvature.

One important property of model space $({\bf H}, g_{\bf H})$ is that we can define new coordinates $(\rho, \Phi)$ called {\it homogeneous coordinates}: Let $\rho$ be the same as the original one and 
 $
 \Phi=\rho^3 \phi.
$ In these homogeneous coordinates, the metric is given by
\begin{equation} \label{hcsm}
g_{\bf H}= 
\left(
\begin{matrix}
1+9\Phi^2\rho^{-2} & -3\rho^{-1}\Phi \\
-3\rho^{-1}\Phi & 1\\
\end{matrix}
\right).
\end{equation}
The homogeneous coordinates are used to define a scaling map, $P \mapsto \lambda P$. More precisely, for $P \in \bf H$ given by $P=(\rho, \Phi)$ in homogeneous coordinates, 
\begin{equation} \label{hcs}
\lambda P= (\lambda \rho, \lambda \Phi).
\end{equation}
Extend  the scaling map to $\bms$ by defining $\lambda P_0=P_0$. From (\ref{hcsm}), the local expression of $g_{\bf H}$ is invariant under this scaling map on $\overline{\bf H}.$ Then, in homogeneous coordinates,
\begin{equation} \label{hcsd}
 d_{\overline{\bf H}}(\lambda P, \lambda Q) =\lambda d_{\overline{\ms}}(P,Q).
\end{equation}

\subsection{Symmetric Geodesics}\label{2.2}
Let 
$
\gamma:(-\infty,\infty) \rightarrow {\bf H}
$
be an arclength parameterized geodesic and  $\gamma_\rho, \gamma_\phi$ be the coordinate functions of $\gamma$ with respect to the standard model space coordinates $(\rho,\phi)$.
The geodesic equations are given by
\begin{eqnarray} \label{gequ}
\gamma_{\rho} \gamma''_{\rho} = 3 \gamma_{\rho}^6 |\gamma'_{\phi}|^2 \ \mbox{ and }  \ \gamma_{\rho}^4 \gamma''_{\phi} = -6 \gamma'_{\rho} \cdot \gamma_{\rho}^3 \gamma'_{\phi}.
\end{eqnarray}   

\begin{definition}
An arclength parameterized geodesic  $\gamma=(\gamma_\rho,\gamma_\phi)$   is said to be a {\it symmetric geodesic}  if
\[
\gamma_{\rho}(s)=\gamma_{\rho}(-s) \ \mbox{ and } \ \gamma_{\phi}(s)=-\gamma_{\phi}(-s).
\] 
\end{definition}
A symmetric geodesic is uniquely determined by its value at $0$.  More precisely, for a fixed $\rho>0$,
there exists a unique symmetric geodesic
\begin{equation} \label{gammarho}
\gamma:\R \rightarrow {\ms}, \ \ \gamma(0)=(\rho,0).
\end{equation}

In homogeneous coordinates, (\ref{gequ}) is rearranged as
\begin{equation}\label{gequ1}
    \gamma_{\rho}''=3 \frac{\left| \gamma_{\Phi}' \gamma_{\rho}^3-3 \gamma_{\Phi}\gamma_{\rho}^2 \gamma_{\rho}'  \right|^2}{\gamma_{\rho}^7} \ \text{ and } \ 6\gamma_{\Phi}(\gamma_{\rho}')^2=\gamma_{\Phi}''\gamma_{\rho}^2-3\gamma_{\Phi}\gamma_{\rho}\gamma_{\rho}''.
\end{equation} Then, given a symmetric geodesic $\gamma=(\gamma_{\rho},\gamma_{\Phi}),$ the scaling curve $\lambda \gamma=(\lambda \gamma_{\rho}, \lambda \gamma_{\Phi})$ also satisfies (\ref{gequ1}). In other words, the scaling of a symmetric geodesic is still a symmetric geodesic.

\begin{definition} \label{Gam}
For $\rho>0$, the image $\Gamma_\rho$ of the parameterized geodesic (\ref{gammarho}) separates $\ms$ into two convex subsets, one of which contains the point $P_0$ in its metric completion.  The closure of the other convex subset will be denoted ${\ms}[\rho]$.   \end{definition}

\begin{lemma} \label{geomH}
The convex subsets  $\ms[\rho]$  satisfy the following properties:  
\begin{itemize}
\item[(a)] $\ms[\rho_2] \subseteq \ms[\rho_1]$ whenever $\rho_1 \leq \rho_2$.
\item[(b)]
$
 {\bf H}[ \lambda  \rho]=  \lambda {\bf H}[\rho]$ for $\lambda>0$.
\end{itemize}
\end{lemma}

\begin{proof}
The assertion (a) is straightforward.  For (b), let $\gamma$ be the symmetric geodesic of  (\ref{gammarho}).  The property of scaling implies that for $t_1, t_2 \in \R$, 
\[
 d_{\overline{\bf H}}(\lambda \gamma(t_1), \lambda \gamma(t_2))=\lambda d_{\overline{\bf H}}(\gamma(t_1), \gamma(t_2))=\lambda \left|t_1-t_2\right|
\]
in homogeneous coordinates. By $\lambda P_0=P_0$ and (\ref{hcsd}),
\[
d_{\overline{\ms}}(\lambda \gamma(0),P_0)=\lambda d_{\overline{\ms}}( \gamma(0),P_0)=\lambda \rho.
\]
Thus, the curve $t \mapsto c(t):=\lambda \gamma(\frac{t}{\lambda})$ is the unit speed parameterization of the symmetric geodesic with initial value $c(0)=(\lambda \rho,0)$ which implies the assertion (b).
\end{proof}

\adjustbox{center}{
\begin{tikzpicture}[
    scale=0.8,
    >=Stealth,
]

\begin{scope}[xshift=7cm, local bounding box=rightFig]
    \draw[->] (0,-4) -- (0,4) node[above] {$\phi$};
    \draw[->] (0,0) -- (5,0) node[right] {$\rho$};
    
    \draw[] (0,-4) -- (0,4);
    \node[] at (-0.5,0) [rotate=90] {$P_0$};
    
    \draw[] (2.5, -5) -- (2.5, 5);
    \node[] at (2.8, -0.3) {$r$};
    \draw[]
      (4.7,4) to[out=183, in=60, looseness=.5] (0.4,3)
      to[out=-130, in=90, looseness=.5] (0.2,1.5)
      to[out=-90, in=90, looseness=.5] (0.2,-1.5)
      to[out=-90, in=130, looseness=.5] (0.4,-3)
      to[out=-60, in=-183, looseness=.5] (4.7,-4)
      node[] at (5, 4.3) {$\Gamma_{\rho_0/2}$};
        
    \draw[]
      (4.3,3.5) to[out=183, in=60, looseness=.5] (1,2.7)
      to[out=-130, in=90, looseness=.5] (.8,1.1)
      to[out=-90, in=90, looseness=.5] (0.8,-1.1)
      to[out=-90, in=130, looseness=.5] (1,-2.7)
      to[out=-60, in=-183, looseness=.5] (4.3,-3.5)
      node[] at (4.8, 3.5){$\Gamma_{\rho_0}$};
\end{scope}

    \node[align=center, below] at (current bounding box.south) 
    {\textbf{Figure 1}};
\end{tikzpicture}
}

\begin{lemma} \label{cvxsets}
For any $r>0,$ 
\[
\lim_{\rho_0 \to 0} \, d_{\overline{\bf H}}(\Gamma_{\rho_0},\Gamma_{\rho_0/2} \backslash B_r(P_0))=r.
\]
\end{lemma}

\begin{proof}
Let $r>0.$ Define $\gamma^{\rho_0}, \, \gamma^{\rho_0/2}$ to be symmetric geodesics such that
\[
\gamma_{\rho}^{\rho_0}(0)=\rho_0, \ \gamma^{\rho_0/2}_{\rho}(0)=\rho_0/2.
\]
Denote their images by $\Gamma_{\rho_0} \text{ and } \Gamma_{\rho_0/2}$ respectively. For each positive $\rho_0<r,$ choose $s_1,s_2 >0$ such that 
\[
\gamma_{\rho}^{\rho_0}(s_1)=r \text{ and } \gamma_{\rho}^{\rho_0/2}(s_2)=r.
\]
Since $\gamma^{\rho_0} \text{ and } \gamma^{\rho_0/2}$ are arclength parameterized geodesics and
\[
\lim_{\rho_0 \to 0} \, d_{\overline{\bf H}}(\gamma^{\rho_0}(s_1), P_0)=\lim_{\rho_0 \to 0} \, d_{\overline{\bf H}}(\gamma^{\rho_0/2}(s_2),P_0)=r,
\]
then 
\begin{equation} \label{finite}
|s_1-s_2| \to 0 \text{ as } \rho_0 \to 0.    
\end{equation}

Then, by applying homogeneous coordinates and (\ref{finite}):
\begin{align*}
    \liminf_{\rho_0 \to 0} \left|\gamma^{\rho_0}_{\phi}(s_1)-\gamma^{\rho_0/2}_{\phi}(s_2) \right| =& \liminf_{\rho_0 \to 0} \left|\frac{\gamma_{\Phi}^{\rho_0}}{\left(\gamma_{\rho}^{\rho_0}\right)^3}(s_1)-\frac{\gamma_{\Phi}^{\rho_0/2}}{\left(\gamma_{\rho}^{\rho_0/2}\right)^3}(s_2)\right|\\
    =& \liminf_{\rho_0 \to 0} \left|\frac{\rho_0\gamma_{\Phi}^1}{\left(\rho_0\gamma_{\rho}^1\right)^3}(s_1)-\frac{\frac{\rho_0}{2}\gamma_{\Phi}^1}{\left(\frac{\rho_0}{2}\gamma_{\rho}^1\right)^3}(s_2)\right|\\
    =& \liminf_{\rho_0 \to 0} \left|\frac{1}{\rho_0^2} \frac{\gamma_{\Phi}^1}{\left(\gamma_{\rho}^1\right)^3}(s_1)- \frac{4}{\rho_0^2}\frac{\gamma_{\Phi}^1}{\left(\gamma_{\rho}^1\right)^3}(s_2)\right|\\
    =& \liminf_{\rho_0 \to 0} \frac{1}{\rho_0^2} \left|\frac{\gamma_{\Phi}^1}{\left(\gamma_{\rho}^1\right)^3}(s_1)- 4\frac{\gamma_{\Phi}^1}{\left(\gamma_{\rho}^1\right)^3}(s_2)\right|\\
    =& \liminf_{\rho_0 \to 0} \frac{1}{\rho_0^2}\left|\gamma_{\phi}^1(s_1)-4\gamma_{\phi}^1(s_2)\right|\\
    =& \ \infty,
\end{align*}
which implies that  
\[
\liminf_{\rho_0 \to 0} |\phi_2-\phi_1|= \infty
\]
where $(\rho_1,\phi_1) \in \Gamma_{\rho_0} \text{ and } 
(\rho_2,\phi_2)\in \Gamma_{\rho_0/2}\backslash B_r(P_0):= \{\gamma^{\rho_0/2}(s): s \geq s_2\}.$ See Figure 1. Observe that  $\Gamma_{\rho_0}\backslash B_r(P_0) \cap \Gamma_{\rho_0/2}\backslash B_r(P_0)=\emptyset.$ So we have 
\begin{align*}
&\lim_{\rho_0 \to 0} \, d_{\overline{\bf H}} (\Gamma_{\rho_0}\backslash B_r(P_0),\Gamma_{\rho_0/2}\backslash B_r(P_0)) \\
=&\lim_{\rho_0 \to 0} \, d_{\overline{\bf H}}(\Gamma_{\rho_0} \backslash B_r(P_0), P_0)+ d_{\overline{\bf H}}(\Gamma_{\rho_0/2} \backslash B_r(P_0),P_0)\\
=& \lim_{\rho_0 \to 0} \left( \inf_{(\rho_1,\phi_1) \in \Gamma_{\rho_0}\backslash B_r(P_0)} \left| \rho_1-0\right| + \inf_{(\rho_2,\phi_2) \in \Gamma_{\rho_0/2}\backslash B_r(P_0)} \left| \rho_2-0\right| \right)\\
=& r+r =2r.
\end{align*} 
Analogously to the argument above, since $\Gamma_{\rho_0} \cap \Gamma_{\rho_0/2} \backslash B_r(P_0)=\emptyset,$ therefore we have the conclusion:
\begin{align*}
    \lim_{\rho_0 \to 0} \, d_{\overline{\bf H}}(\Gamma_{\rho_0},\Gamma_{\rho_0/2}\backslash B_r(P_0)) =& 
     \lim_{\rho_0 \to 0} \, d_{\overline{\bf H}}(\Gamma_{\rho_0}, P_0)+ d_{\overline{\bf H}}(\Gamma_{\rho_0/2}\backslash B_r(P_0),P_0)\\
    =&\lim_{\rho_0 \to 0} \left( \inf_{(\rho_1,\phi_1) \in \Gamma_{\rho_0}} \left| \rho_1-0\right| + \inf_{(\rho_2,\phi_2) \in \Gamma_{\rho_0/2}\backslash B_r(P_0)} \left| \rho_2-0\right| \right)\\
    =&\lim_{\rho_0 \to 0} \rho_0 + r\\
    =& \ r.
\end{align*}
\end{proof}

\begin{lemma} \label{distest}
If $C$ is the the complement 
of  ${\ms}[\rho/2] \cup B_r(P_0) $, 
then 
\[
d_{\overline{\ms}} (C, \Gamma_\rho) \geq d_{\overline{\ms}} (\Gamma_{\rho/2} \backslash B_r(P_0), \Gamma_\rho)
\]
\end{lemma}

\begin{proof}
Since 
\[
\partial C = \left( \Gamma_{\rho/2} \backslash B_r(P_0) \right) \cup  \Phi
\]
where $\Phi:=\{\rho=r\} \backslash \ms[\rho/2]$,
we have
\[
d_{\overline{\ms}} (C, \Gamma_\rho) = \min \{d_{\overline{\bf H}}(\Phi, \Gamma_\rho), d_{\overline{\ms}} (\Gamma_{\rho/2} \backslash B_r(P_0), \Gamma_\rho)\}.
\]
Thus, the assertion follows from the fact that 
\[
d_{\overline{\bf H}}(\Phi, \Gamma_\rho) = d_\ms(\Phi \cap \Gamma_{\rho/2}, \Gamma_\rho) \geq d_\ms(\Gamma_{\rho/2} \backslash B_r(P_0), \Gamma_\rho).
\]
\end{proof}

\subsection{Metric Space $\overline{\bf H}_A$}\label{2.3}
We now define a {\it metric space} introduced in \cite{9}. Let $\overline{\bf H}_{\nu}$ be a copy of $\overline{\bf H}$ for each $\nu \in A$ where $A$ is a finite set. Define
\begin{equation}
    \label{H_A0}
    \overline{\bf H}_A:=  \amalg_{\nu \in A} \overline{\bf H}_{\nu}\ / \sim,
\end{equation}
where $\sim$ identifies all boundary points $P_0$ in $\overline{\bf H}_{\nu}$ as a single point. $\overline{\bf H}_A$ is endowed with the distance function $d_A$: For any $x=(\rho,\phi),y=(\rho',\phi')$ in $\overline{\bf H}_A,$ 
\begin{equation}\label{d_A}
d_{A}(x,y)=
\begin{cases}
    d_{\overline{\bf H}}(x,y) & x,y \in \overline{\bf H}_{\nu}\\
    \rho+\rho'  & x \in {\overline{\bf H}_{\nu}}, \, y \in \overline{\bf H}_{\nu'} \text{ for } \nu \neq \nu'.
\end{cases}
\end{equation}
The geodesic in $\overline{\bf H}_A$ connecting $x \in \overline{\bf H}_{\nu}$ and $y \in \overline{\bf H}_{\nu'}$, for $\nu \neq \nu',$ is the union of horizontal segments from $x=(\rho,\phi)$ to $P_0$ and from $y=(\rho',\phi')$ to $P_0.$ (cf. \cite[Section 2]{9})

Since $\overline{\bf H}$ is the metric completion of NPC space and $\{P_0\}$ is a convex subset, \cite[Theorem 2.11.1]{10} implies that $\overline{\bf H} \amalg \overline{\bf H} / \sim,$ which $\sim$ is induced by the identity map $id: \{P_0\} \to \{P_0\},$ is an NPC space. Inductively, we can also prove that $\overline{\bf H}_A$ is an NPC space.

\subsection {Harmonic Map to NPC Space} 
For map $u: (\Omega,g) \to X$ where $X$ is an NPC space, the $\epsilon$-energy density function is defined in \cite[Section 1.2]{2} as
\[
e_{\epsilon}(x) = 
\begin{cases} 
\displaystyle\int_{y \in \partial B_{\epsilon}(x)} \frac{d^2(u(x), u(y))}{\epsilon^2} \frac{d\sigma}{\epsilon^{n-1}}, & x \in \Omega_{\epsilon} \\ 
\mathmakebox[\widthof{$\displaystyle\int\limits_{y \in \partial B_{\epsilon}(x)} \frac{d^2(u(x), u(y))}{\epsilon^2} \frac{d\sigma}{\epsilon^{n-1}}$}][c]{0}, & \text{otherwise}
\end{cases}
\]
where $d \sigma$ here is $n-1$ dimensional surface measure and $\Omega_{\epsilon}:=\{x \in \Omega: dist(x,\partial \Omega) \geq \epsilon\}.$
Say $u$ has finite energy if 
\[ E^{u}:=\sup_{\varphi \in C_c(M), 0 \leq \varphi \leq 1} \limsup_{\epsilon \to 0} \int_{\Omega} \varphi e_{\epsilon} \, dvol_g < \infty.\]
From the result in \cite[Section 1.5]{2}, we know that as $\epsilon \to 0, \, e_{\epsilon}(x) \, dvol_g$ converges weakly to a Sobolev energy density measure $|du|^2(x) dvol_g$ weakly. This defines the energy formula in $\Omega$:
\[ E^{u}[\Omega]:=\int_{\Omega} |du|^2 dvol_g.\]
We say a continuous map $u:\Omega \to X$ is {\it harmonic} if it's the locally energy minimizing map i.e. for any $p \in \Omega,$ there exists $r>0$ such that the restriction map $u |_{B_r(p)}$ is the energy minimizer among all admissible maps in the space $W^{1,2}_u(B_r(p),X):=\{h \in W^{1,2}(B_r(p),X):d(u,h) \in W_0^{1,2}(B_r(p))\}$ (cf. \cite[Section 2.2]{2}). Moreover, a harmonic map $u$ is Lipschitz continuous by the following.

\begin{theorem}[Theorem 2.4.6 in \cite{2}] \label{cont}   
Let $\Omega$ be a Lipschitz Riemannian domain, and let $u$ solve the Dirichlet Problem. Then $u$ is locally Lipschitz continuous in the interior of $\Omega.$
\end{theorem}

A nonconstant harmonic map $u: \Omega \rightarrow \overline{\bf H}$ has the following important monotonicity formula. Given  $x_0 \in \Omega$ and $r>0$ such that $B_{r}(x_0) \subset \Omega$, let 
\begin{eqnarray*} \label{nott0}
E^u(r) :=\int_{B_{r}(x_0)} |\nabla u|^2 d\mu \ \mbox{ and } \
I^u(r) := \int_{\partial B_{r}(x_0)} d^2(u(x),u(x_0)) d\Sigma.
\end{eqnarray*}
There exists a constant $c>0$ depending only on the $C^2$ norm of the domain metric $g$ (with $c=0$ when $g$ is the standard Euclidean metric) such that
\begin{equation} \label{monotonicityformula0}
r \mapsto e^{c r^2}\frac{r \ E^u(r) }{I^u(r)}, \ \ \  \ \ r \mapsto e^{c r^2}\frac{I^u(r) }{r^{n+1}} \ \ \ 
 \end{equation}
are non-decreasing.  
Recall that $I^u(r)>0$ for any $r>0,$ which follows from the fact that $d^2(u(x),u(x_0))$ is subharmonic and the Mean Value Property for subharmonic function. As a non-increasing limit of continuous functions,
\[
Ord^u(x_0):=\lim_{r \rightarrow 0} e^{c r^2}\frac{r \ E^u(r) }{I^u(r)}
\]
 is an upper semicontinuous function.  The value $Ord^u(x_0)$ is called the order of $u$ at $x_0.$  

\subsection {Convergence in Pullback Sense} \label{pullback}
We use the same notation as \cite[Section 3]{3} and summarize the idea. Let $\Omega_0=B_1(0)\text{ and } u: \Omega_0 \to X$ as above. Define $d_0$ to be the pullback pseudodistance on $\Omega_0 \times \Omega_0$ induced from $u,$
\[ d_0(x,y):=d(u(x),u(y)).\]
Inductively let $\Omega_{i+1}:=\Omega_i \times \Omega_i \times [0,1]$ with inclusion $\Omega_i \hookrightarrow \Omega_{i+1}$ by $x \mapsto (x,x,0).$ Extend $u_i$ to $u_{i+1}: \Omega_{i+1} \to X$ by 
\[u_{i+1}(x,y,\lambda):=(1-\lambda)u_i(x)+\lambda u_i(y),\]
and let $d_{i+1}$ denote the corresponding pullback pseudodistance. Define $\Omega_{\infty}=\bigcup \Omega_i$ and equip $\Omega_{\infty} \times \Omega_{\infty}$ with a pseudodistance $d_{\infty}$ whose restriction $d_i$ on $\Omega_i \times \Omega_i$ satisfies the inequality:
    \begin{equation}\label{npcineq}
    d^2_{i+1}(z,(x,y,\lambda)) \leq (1-\lambda)d^2_{i+1}(z,(x,x,0))+\lambda d^2_{i+1}(z,(y,y,0))-\lambda (1-\lambda)d^2_i(x,y)
    \end{equation}
where $x,y \in \Omega_i, \, z \in \Omega_{i+1}$ and $\lambda \in [0,1].$ Define its metric completion $Z := \Omega_{\infty} / \sim$ with the equivalence relation that $x \sim y$ if and only if $d_{\infty}(x,y)=0.$ Inequality (\ref{npcineq}) implies that $(Z,d_{\infty})$ is an NPC space. The pullback metric setting implies that the convex hull of $u(\Omega)$ is isometric to the quotient metric space $Z= \Omega_{\infty}/ \sim.$  
    
Given a sequence of blow-up maps $\{u_k=u_{\sigma_k}: \Omega_0 \to (X,d_k)\}$ into NPC spaces, iteratively construct $u_{k,i+1}: \Omega_{i+1} \to X_k=(X,d_k)$ induced from $u_{k,i}: \Omega_i \to X_k$ by
\[u_{k,i+1}(x,y,\lambda)=(1-\lambda)u_{k,i}(x)+\lambda u_{k,i}(y).\]
Then, the pullback pseudodistance $d_{k,i}$ of $u_{k,i}$ on $\Omega_i \times \Omega_i$ inherits inequality (\ref{npcineq}) from the NPC property of $X_k.$ For each $k,$ define $d_{k,\infty}$ by the restriction $d_{k,\infty}|_{\Omega_i \times \Omega_i}:=d_{k,i}.$ Say {\it $u_k$ converges locally uniformly to $u_*:\Omega_0 \to X_*=(X,d_*)$ in the pullback sense} if the pullback pseudodistance $d_{k,\infty}$ converges to $d_{*,\infty}$ locally uniformly i.e. $d_{k,i}$ converges to $d_{*,i}$ uniformly in each compact subset of $\Omega_i \times \Omega_i.$ Here, target space $X_*$ is isometric to the metric completion $Z:=\Omega_{\infty}/\sim$ where $x \sim y$ if and only if $d_{*,\infty}(x,y)=0.$ (cf. \cite[Section 3]{3})

\subsection{Local coordinates near $\overline{\mathcal{T}}$ with Weil--Petersson Metric Completion of $\mathcal{T}$}
Let $\mathcal{T}$ denote the Teichmüller space of an oriented compact surface of genus $g$ with $p$ marked points. Equipped with the Weil--Petersson metric $g_{wp},$ $(\mathcal{T},g_{wp})$ is a smooth Kähler manifold of complex dimension $k=3g-3+p>0$ with negative sectional curvature. Its Weil--Petersson metric completion $(\overline{\mathcal{T}},d_{wp})$ is a stratified NPC metric space. In particular, $\overline{\mathcal{T}}$ is decomposed as:
\[ \overline{\mathcal{T}}=\bigcup \mathcal{T}'.
\]
Here, $\mathcal{T}'$ is a $j$-dimensional open stratum parameterizing nodal surfaces obtained by pinching $k-j$ mutually disjoint simple closed curves to nodes. Teichmüller space $\mathcal{T}$ itself is a $k$-dimensional open stratum. Each open stratum is a product of lower-dimensional Teichmüller spaces and is totally geodesic with respect to Weil--Petersson metric.

For a boundary point $P \in \mathcal{T}' \subset \overline{\mathcal{T}}$ in a $j$-dimensional stratum, which corresponds to a nodal surface $S_0$, {\it local coordinates} in the neighborhood near $P$ can be constructed as follow: Let $r=(r_1,...,r_j) \in \mathbb{C}^j$ parametrize the neighborhood of nodal surface $S_0$ in $\mathcal{T}'$ and the plumbing coordinates $t=(t_1,...,t_{k-j}) \in \mathbb{C}^{k-j}$ regularize the nodes. With positive $t_i, \, i=1,...,k-j,$ we have an analytic family of Riemann surfaces $S_{r,t}$ of genus $g$ with $p$ marked points of $S_0.$ When $(t_1,...,t_{k-j}) \to (0,...,0),$ the Riemann surface degenerates to the nodal surface $S_r.$ Combined together, $r \text{ and } t$ define local coordinates on $\overline{\mathcal{T}}$ near $P$ (cf. \cite[Sections 1 and 2]{7}). 

The parameter $t=(t_1,...,t_{k-j})$ induces a model space ${\bf H}^{k-j},$ where $t_i$ maps to $(\rho_i,\phi_i) \in \bf{H}$ via:
\[ \rho_i=2(-\log |t_i|)^{-\frac{1}{2}} \text{ and } \phi_i=\frac{1}{8}\arg t_i.
\]
Specifically, for $P \in \mathcal{T}',$ where $\mathcal{T}'$ is a $j$-dimensional stratum, there exists a neighborhood $N \subset \overline{\mathcal{T}}$ of $P,$ a neighborhood $\mathcal{U} \subset \mathbb{C}^j$ of $0,$ a neighborhood $\mathcal{V} \subset \overline{\bf H}^{k-j}$ of $P_0,$ and an injection derived from the previous mappings
\begin{equation}\label{hom}
F: N \to \mathcal{U} \times \mathcal{V} \subset \mathbb{C}^j \times \overline{\bf H}^{k-j} \text{ by } Q \mapsto (r_1,...,r_j,(\rho_1,\phi_1),...,(\rho_{k-j},\phi_{k-j}))
\end{equation}where $F(P)=(0,...,0,P_0,...,P_0) \in \mathbb{C}^j \times \overline{\bf H}^{k-j}.$ $F$ is a homeomorphism and a biholomorphism when its domain is restricted on the open stratum (cf. \cite[Section 2.2]{1}).

Moreover, let $G$ be the smooth pullback metric extension of $g_{wp}$ on $\mathbb{C}^j$ under $F^{-1}$ and $h$ be the metric on $\overline{\bf H}^{k-j}$ defined in section \ref{sec:modelspace}. The tensor $G \otimes h$ will be the product metric on $\C^j \times {\overline{\bf H}}^{k-j}.$ We have $g_{wp} - G \otimes h \to 0$ in $C^1$ in terms of the complex parameter $t=(t_1,...,t_{k-j})$ given by (\ref{hom}). The precise estimates are contained in \cite{6}.

\section{Harmonic Maps into Model Space} \label{sec:blowup0}

In this section, we prove that a nonconstant harmonic map into the metric completion of model space has no singularities. We define the \emph{singular set} as
\[
\mathcal S(u)=\{x \in \Omega:  u(x)=P_0\}.
\]
A \emph{singular point} is a point in $\mathcal{S}(u)$ and a \emph{regular point} is a point that is not a singular point. 
\begin{theorem}\label{result0}
If $u:\Omega \rightarrow \overline{\bf H}$ is a nonconstant harmonic map, then $u$ has no singular points. 
\end{theorem}

This section is devoted to the proof of Theorem \ref{result0}. On the contrary, we assume the singular set of $u$ is non-empty. Observe that $\mathcal{S}(u)$ is a closed set because $u$ is continuous from Theorem \ref{cont}.
In a neighborhood of  $x \in \Omega \backslash \mathcal S(u)$, $u$ maps into a smooth Riemannian manifold ${\bf H}$, and we can write
\[
u=(u_{\rho},u_{\phi})
\]
in terms of coordinates $(\rho,\phi)$.

Let  $x_0 \in \partial \mathcal{S}(u)$ and  
\[
\alpha :=Ord^u(x_0)>0.
\]
For $r_0>0$ such that $B_{r_0}(x_0) \subset \Omega$, identify $(B_{r_0}(x_0),g) \subset \Omega$ with the Euclidean ball $B_{r_0}(0) \subset \R^n$ via normal coordinates centered at $x_0=0$. Let $u:(B_{r_0}(0),g) \rightarrow \bms$ be the restriction of $u$. We construct the sequence $\{u_{\sigma_i}\}$ of the blow-up maps of $u$ at $x_0$: Define a function $\lambda^u: (0,r_0] \rightarrow (0,\infty)$ by
\[
\lambda^u(\sigma) = \left(\sigma^{1-n} \int_{\partial B_\sigma(0)} d^2(u,u(0)) d\Sigma \right)^{-\frac{1}{2}}.
\]
For $\sigma \in (0,r_0]$, the {\it blow-up map} of $u$ at $x_0$ is given by
\[
u_{\sigma}:(B_1(0),g) \rightarrow \bms, \ \ \ u_{\sigma}(x)=\lambda^u(\sigma) u(\sigma x),
\] where $u_{\sigma}(0)=P_0$ for any $\sigma >0$ and $\lambda P$ is defined as in (\ref{hcs}).
Notice that $u_{\sigma}$ is harmonic since harmonicity is invariant under scaling and monotonicity property (\ref{monotonicityformula0}) implies
$$
2 ^2e^{c/4}\left(\lambda^{u_{\sigma}}\left(\frac{1}{2}\right)\right)^{-2} = 2^{n+1}e^{c/4} \int_{\partial B_{\frac{1}{2}}(0)} d^2_{\bf H}(u_{\sigma},u(0)) d\Sigma \leq e^c\int_{\partial B_1(0)} d^2_{\bf H}(u_{\sigma},u(0)) d\Sigma =e^c
$$
For domain metrics $g$ sufficiently close to Euclidean metric, i.e. for $c$ close to $0,$
\begin{equation} \label{lowerbdlambda0}
1 \leq  \lambda^{u_{\sigma}}\left(\frac{1}{2}\right).
\end{equation}

At this point, we need a tangent map satisfying particular properties. To that end, we produce a sequence $\sigma_i \to 0$ and an nonconstant homogeneous harmonic tangent map $u_*$ following the idea of Appendix I. Initially, $u_*:B_1(0) \to \overline{\bf H}_*$ is not good enough for our purposes since $\overline{\bf H}_*$ is only an abstract NPC space. In Appendix I, we show that in fact we can modify the target so that $u_*$ maps into a concrete NPC space $(\overline{\bf H}_A,d_A)$ defined in section \ref{2.3} and
\begin{equation}\label{distapprox0}
    d_{\overline{\bf H}}(u_{\sigma_i}(\cdot),u_{\sigma_i}(\cdot)) \to d_A(u_*(\cdot),u_*(\cdot)) \text{ uniformly on compact subsets of $B_1(0).$}
\end{equation}
Moreover, $u_*$ is {\it piecewise a function} in the sense of Definition \ref{pw0}.
Here,  $A$ is defined as follows: Let  $\Omega_1,...,\Omega_k$ be the connected components in $B_1(0) \setminus \{x \in B_1(0): u_*(x)= u_*(0)\}.$ Then, $A$ is the set of equivalence classes of $\{1,...,k\}$ such that $\nu \sim \nu'$ if for every pair of points $x \in \Omega_{\nu}$ and $y \in \Omega_{\nu'},$
\begin{equation}\label{A}
d_A(u_*(x),u_*(y)) <  d_A(u_*(x),u_*(0))+d_A(u_*(y),u_*(0)).
\end{equation}
As shown in the proof of Lemma \ref{H_A} in Appendix I, $u_*(\Omega_{\nu})$ and $u_*(\Omega_{\nu'})$ are contained in the same copy of model space $\overline{\bf H}$ and $\big| u_{\sigma_i}^{\phi}(x)-u_{\sigma_i}^{\phi}(y) \big|$ is bounded independent of $\sigma_i$ for $x \in \Omega_{\nu}$ and $y \in \Omega_{\nu'}.$ Note that $|A| \geq 2,$ which is shown in Lemma \ref{|A|}. 

\begin{remark}\label{2pts}
    Lemmas \ref{approx0} -- \ref{s30} below only rely on the fact that $u_*$ is an nonconstant homogeneous harmonic map and piecewise a function and the distance convergence (\ref{distapprox0}).
\end{remark}

Fix a point $x_m \in \Omega_m$ for $m=1,...,k.$ By taking subsequence if necessary and renumbering $\Omega_1,... \Omega_k,$ we can assume
\[
\max_{m=1,...,k}u_{\sigma_i}^{\phi}(x_m)=u_{\sigma_i}^{\phi}(x_k) \geq u_{\sigma_i}^{\phi}(x_{k-1}) \geq ... \geq\min_{m=1,...,k}u_{\sigma_i}^{\phi}(x_m)=u_{\sigma_i}^{\phi}(x_1).
\]
Define an isometry $T_{c_i}: \overline{\bf H} \to \overline{\bf H}$ by setting 
\[
T_{c_i}(P_0)=P_0 \text{ and } T_{c_i}(\rho,\phi)=(\rho,\phi-{c_i}),
\]
where $c_i=\frac{u_{\sigma_i}^{\phi}(x_k)+u_{\sigma_i}^{\phi}(x_1)}{2}.$ Then for all $\sigma_i$'s and corresponding $c_i$'s, 
\[
(T_{c_i} \circ u_{\sigma_i})^{\phi}(x_k)=-(T_{c_i} \circ u_{\sigma_i})^{\phi}(x_1).
\]
\begin{definition}\label{normalized}
By post-composing with this translation, we can assume that the sequence $\{u_{\sigma_i}\}$ satisfies the normalization 
\[
u_{\sigma_i}^{\phi}(x_k)=-u_{\sigma_i}^{\phi}(x_1).
\]We will call these maps the {\it normalized blow-up maps}.
\end{definition}

Next, we define a sequence $\{L_i\}$ from the sequence $\{u_{\sigma_i}\}$ of normalized blow-up maps: First define
\[
L_{m,i}: \Omega_m \to {\bf H}, \ \ \ L_{m,i}(x) = (d_A(u_*(x),u_*(0)),  u_{\sigma_i}^{\phi}(x_m))
\]
and then define
\begin{equation}\label{L_i}
L_i : B_1(0) \to \overline{\bf H}, \ \ \
L_i(x)=
\begin{cases}
L_{m,i}(x) & x \in \Omega_m
\\
P_0 & x \in u_*^{-1}(u_*(0)).
\end{cases}
\end{equation}

\adjustbox{center}{
\begin{tikzpicture}[
    scale=0.7,
    >=Stealth,
    declare function={
        gamma(\x) = 0.15*\x^2; 
    }
]
    \begin{scope}[local bounding box=Fig1]
        \draw[->] (0,-5) -- (0,5) node[above] {$\phi$};
        \draw[->] (0,0) -- (6,0) node[right] {$\rho$};
        \node at (0,-4.5) [below right] {$P_0$};
        
        \draw[black] (0,-4) -- (0,4);
        
        \draw[black] (0,-4) -- (4.5,-4) node[right] {$L_{1,i}$};
        \draw[black] (0,-2) -- (4.5,-2) node[right] {$L_{3,i}$};
        \draw[black] (0,-3) -- (4.5,-3) node[right] {$L_{2,i}$};
        \draw[black] (0,1) -- (4.5,1) node[right] {$L_{k-1,i}$};
        \draw[black] (0,4) -- (4.5,4) node[right] {$L_{k,i}$};
        
        \foreach \y in {-0.28, 0.28} {
            \node at (-0.5,\y) {$\vdots$};
        }
    \end{scope}
    \node[align=center, below] at (Fig1.south) {
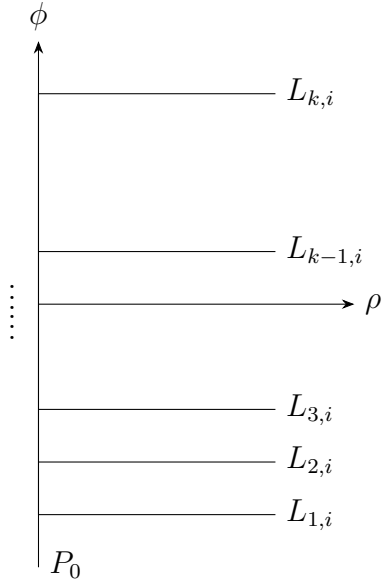
    \text{The image of map $L_i$}
    };
    \node[align=center, below] at (current bounding box.south) 
    {\textbf{Figure 2}};
\end{tikzpicture}
}\\

\begin{lemma} \label{approx0}
The map $L_i$ defined above satisfies
\[
d_{\overline{\bf H}}(L_i(\cdot),L_i(\cdot))-d_{\overline{\bf H}}(u_{\sigma_i}(\cdot),u_{\sigma_i}(\cdot)) \to 0 \text{ as } \sigma_i \to 0
\]
uniformly on compact sets of $B_1(0).$
\end{lemma} 
\begin{proof}
We claim that for $x \in \Omega_s$ and $y \in \Omega_t$ where $s \nsim t,$
\begin{equation} 
    \label{phicoords0}
    \lim_{i \to \infty}|u^{\phi}_{\sigma_i}(x)-u^{\phi}_{\sigma_i}(y)| = \infty.
\end{equation}
Suppose on the contrary that $|u^{\phi}_{\sigma_i}(x)-u^{\phi}_{\sigma_i}(y)|$ is bounded as $i \to \infty \, (\sigma_i \to 0).$ This implies that there exists $\delta>0$ such that for all $i$ sufficiently large, 
\[
d_{\overline{\bf H}}(u_{\sigma_i}(x),u_{\sigma_i}(y)) <d_{\overline{\bf H}}(u_{\sigma_i}(x),u_{\sigma_i}(0))+d_{\overline{\bf H}}(u_{\sigma_i}(y), u_{\sigma_i}(0))-\delta.
\] 
Then, $d_A(u_*(x),u_*(y))< d_A(u_*(x),u_*(0))+d_A(u_*(y),u_*(0)),$ which contradicts the condition that $s \nsim t.$

Let $K\in B_1(0)$ be a compact set and $\epsilon > 0$ arbitrarily small. We can choose a neighborhood $U$ of $u_*^{-1}(P_0) \subset B_1(0)$ and a positive integer $N_1$ satisfying
\begin{equation}
    d_{\overline{\bf H}}(u_{\sigma_i}(x),u_{\sigma_i}(0)) < \frac{\epsilon}{4} \text{ for any } x \in U \text{  and for all }i \geq N_1
\end{equation}
and
\[
d_{\overline{\bf H}}(u_{\sigma_i}(x),u_{\sigma_i}(0)) \geq \frac{\epsilon}{8} \text{ for any } x\notin K \setminus U \text{ and for all } i \geq N_1.
\]

Let $x,y \in K.$ We treat the following three cases separately. \\
\underline{Case 1.} $x,y \in U.$ 

For $i \geq N_1,$
\begin{eqnarray*}
d_{\overline{\bf H}}(u_{\sigma_i}(x),u_{\sigma_i}(y)) & \leq & d_{\overline{\bf H}}(u_{\sigma_i}(x),P_0)+d_{\overline{\bf H}}(u_{\sigma_i}(y),P_0)\\
& < & \frac{\epsilon}{4}+\frac{\epsilon}{4}\\
& = & \frac{\epsilon}{2}\\
d_{\overline{\bf H}}(L_i(x),L_i(y)) & \leq & d_{\overline{\bf H}}(L_i(x),P_0)+d_{\overline{\bf H}}(L_i(y),P_0)\\
& = & L_i^{\rho}(x)+L_i^{\rho}(y)\\
& = & d_A(u_*(x),u_*(0))+d_A(u_*(y),u_*(0))\\
& \leq & \frac{\epsilon}{4}+\frac{\epsilon}{4}\\
& = & \frac{\epsilon}{2}.\\
\end{eqnarray*}
Thus, for any $x,y \in U,$
\[
|d_{\overline{\bf H}}(L_i(x),L_i(y))-d_{\overline{\bf H}}(u_{\sigma_i}(x),u_{\sigma_i}(y))| < \epsilon.
\]
\underline{Case 2.} $x,y \in K \setminus U$ where $x \in \Omega_s, y \in \Omega_t, s \nsim t.$

For $i \geq N_1, d_{\overline{\bf H}}(u_{\sigma_i}(x),u_{\sigma_i}(0)) \geq \epsilon/8 >0,$ which guarantees that $d_A(u_*(x),u_*(0))$ is bounded away from zero. The fact (\ref{phicoords0}) implies that $\lim_{i \to \infty} d_{\overline{\bf H}}(L_i(x),L_i(y))=d_A(u_*(x),u_*(0))+d_A(u_*(y),u_*(0)).$ Additionally, $u_*(\Omega_s)$ and $u_*(\Omega_t)$ are contained in different copies of $\overline{\bf H},$ which is proved in Appendix I. Thus, (\ref{d_A}) implies that $d_A(u_*(x),u_*(y))=d_A(u_*(x),u_*(0))+d_A(u_*(y),u_*(0)).$ Consequently, there exists an integer $N_2$ large enough and independent of $x,y$ such that for $ i\geq N_2,$ 
\[
\Big|d_{\overline{\bf H}}(L_i(x),L_i(y))-d_A(u_*(x),u_*(y))\Big| < \frac{\epsilon}{2}.  
\]
Furthermore, (\ref{distapprox0}) implies that there is an integer $N_3$ such that for every $i \geq N_3,$
\[
\Big|d_A(u_*(x),u_*(y)) - d_{\overline{\bf H}}(u_{\sigma_i}(x),u_{\sigma_i}(y)) \Big| < \frac{\epsilon}{2}.
\]
Therefore, for all $i \geq \max\{N_2, N_3\},$
\[
    \Big| d_{\overline{\bf H}}(L_i(x),L_i(y))-d_{\overline{\bf H}}(u_{\sigma_i}(x),u_{\sigma_i}(y))\Big| <   \epsilon.
\]
\underline{Case 3.} $x,y \in K \setminus U$ where $x \in \Omega_s, \, y \in \Omega _t$ with $s \sim t.$ 

In this case, $u_*$ maps $\Omega_s$ and $\Omega_t$ to the same copy of $\overline{\bf H}$ and $|u^{\phi}_{\sigma_i}(x)-u^{\phi}_{\sigma_i}(y)|$ is bounded for any $\sigma_i,$ which is proved in the argument of Lemma \ref{H_A}. Recall that $u_*$ is piecewise a function into $\overline{\bf H}_A$ (cf. Lemma \ref{lem:pwfnc0} and Lemma \ref{H_A}). By these facts and (\ref{distapprox0}), we have
\begin{eqnarray*}
   \lim_{i \to \infty} d_{\overline{\bf H}}(L_i(x),L_i(y))&=& \lim_{i \to \infty} d_{\overline{\bf H}}\big((d_A(u_*(x),u_*(0)),u_{\sigma_i}^{\phi}(x_s)), 
(d_A(u_*(y),u_*(0)),u_{\sigma_i}^{\phi}(x_t))\big)\\
&=& \lim_{i \to \infty} d_{\overline{\bf H}}\big((d_A(u_*(x),u_*(0)),u_{\sigma_i}^{\phi}(x)), 
(d_A(u_*(y),u_*(0)),u_{\sigma_i}^{\phi}(y))\big).\\
&=& d_A(u_*(x),u_*(y)).
\end{eqnarray*}
Thus, we can choose $N_4$ to be an integer large sufficiently such that for $i \geq N_4,$
\[
\Big|d_{\overline{\bf H}}(L_i(x),L_i(y))-d_A(u_*(x),u_*(y)) \Big| < \frac{\epsilon}{2}
\]
and
\[
\Big|d_A(u_*(x),u_*(y)) - d_{\overline{\bf H}}(u_{\sigma_i}(x),u_{\sigma_i}(y)) \Big| < \frac{\epsilon}{2}.
\]
Taking the above inequalities together, we have
\[
\Big|d_{\overline{\bf H}}(L_i(x),L_i(y))-d_{\overline{\bf H}}(u_{\sigma_i}(x),u_{\sigma_i}(y))\Big| < \epsilon.
\]
Therefore, for any $\epsilon >0, \text{ we choose } N=\max \{N_1,N_2,N_3, N_4\}$ 
to ensure that for any $x,y \in K$ and for any $i \geq N, \ \left|d_{\overline{\bf H}}(u_{\sigma_i}(x),u_{\sigma_i}(y))-d_{\overline{\bf H}}(L_i(x),L_i(y))\right| < \epsilon.$
\end{proof}

\begin{lemma}\label{u&L0}
    $d_{\overline{\bf H}}(u_{\sigma_i},L_i) \to 0$ uniformly on compact subsets of $B_1(0).$
\end{lemma}
\begin{proof}
Let $K \in B_1(0)$ be a compact set. Lemma \ref{approx0} implies that in $K,$
\begin{eqnarray*}
\lim_{\sigma_i \to 0} d_{\overline{\bf H}}(u_{\sigma_i}(x),P_0) 
& = & \lim_{i \to \infty} d_{\overline{\bf H}}(L_i(x), L_i(0)) \\
& = & \lim_{i \to \infty} d_{\overline{\bf H}}(L_i(x),P_0).
\end{eqnarray*}
Following the similar idea to the proof of Lemma \ref{approx0}, we proceed by cases:\\
\underline{Case 1.}
For $\epsilon > 0$ arbitrarily small, choose a neighborhood $U$ of $u_*^{-1}(P_0) \text{ in } B_1(0)$ such that there exists a positive integer $N_1$ satisfying $d_{\overline{\bf H}}(u_{\sigma_i}(x),P_0) < \frac{\epsilon}{2}$ for any $x \in U \text{  and for all }i \geq N_1.$ Then, let $x\in U,$ 
    \[
    d_{\overline{\bf H}}(u_{\sigma_i}(x),L_i(x)) \leq d_{\overline{\bf H}}(u_{\sigma_i}(x),P_0)+d_{\overline{\bf H}}(L_i(x),P_0) < \ \epsilon.
    \]
\underline{Case 2.}
Let $x \in (K\setminus U) \cap \Omega_m,$ which implies that $d_A(u_*(x),u_*(0))$ is bounded below by $\delta_0>0.$ In the proof of Theorem \ref{lem:pwfnc0}, we show that on $(K \setminus U) \cap \Omega_m,$
\[
u_{\sigma_i}^{\rho}(x) \to d_A(u_*(x),u_*(0)) \mbox{ and } \Big|u_{\sigma_i}^{\phi}(x)-u_{\sigma_i}^{\phi}(x_m) \Big| \to 0.
\]
Thus, for any $\epsilon>0,$ there exists a positive integer $N_2$ such that for all $i \geq N_2,$
\[
d_{\overline{\bf H}}(u_{\sigma_i}(x),L_i(x))=d_{\overline{\bf H}}\big((u_{\sigma_i}^{\rho}(x),u_{\sigma_i}^{\phi}(x)),(d_A(u_*(x),u_*(0)),u_{\sigma_i}^{\phi}(x_m))\big) < \epsilon.
\]

Therefore, pick $N=\max\{N_1,N_2\}.$ For any $\epsilon>0,$ there exists a integer $N$ so that for all $i \geq N,$
\[
d_{\overline{\bf H}}(u_{\sigma_i}(x),L_i(x))< \epsilon \mbox{ for all } x \in K.
\]
\end{proof}

\adjustbox{center}{
\begin{tikzpicture}[
    scale=0.8,
    >=Stealth,
]

\begin{scope}[local bounding box=Fig2]
    \draw[->] (0,-5) -- (0,5) node[above] {$\phi$};
    \draw[->] (0,0) -- (6,0) node[right] {$\rho$};
    \node at (0,-4) [below right] {$P_0$};
    
    \draw[black] (0,-4) -- (0,4);
    
    \draw[black] (0.4,-3) -- (4.7,-3) node[right] {$L_{2,i}$};
    \draw[black] (0.2,-2) -- (4.7,-2) node[right] {$L_{3,i}$};
    \draw[black] (0.2,1) -- (4.7,1) node[right] {$L_{k-2,i}$};
    \draw[black] (0.4,3) -- (4.7,3) node[right] {$L_{k-1,i}$};
    
    \draw[black]
      (4.7,4) node[right]{$\gamma_{\sigma_i}$} to[out=183, in=60, looseness=.5] (0.4,3) 
      to[out=-130, in=90, looseness=.5] (0.2,1.5)
      to[out=-90, in=90, looseness=.5] (0.2,-1.5)
      to[out=-90, in=130, looseness=.5] (0.4,-3)
      to[out=-60, in=-183, looseness=.5] (4.7,-4) ;

        \foreach \y in {-0.25, 0.25, 2.25} {
            \node at (-0.5,\y) {$\vdots$};
        }
\end{scope}

    \node[align=center, below] at (Fig2.south) {
        \text{The Image of $\Lambda_{\sigma_i}$}
    };
    \node[align=center,below] at (current bounding box.south)
    {\textbf{Figure 3}};
\end{tikzpicture}
}

Let $\gamma_{\sigma_i}:\R \to {\bf H}$ be a symmetric geodesic passing through $(1,u_{\sigma_i}^{\phi}(x_1)) \text{ and } (1,u_{\sigma_i}^{\phi}(x_k)),$ and let $\Gamma_{\sigma_i}$ be its image. By \cite[Lemma 3.17]{1}, $d_{\bf H}(\Gamma_{\sigma_i}, \text{Im} \, L_{1,i} \cup \text{Im} \, L_{k,i} \cup P_0) \to 0$ as $\sigma_i \to 0.$ Moreover, since $|A| \geq 2,$  (\ref{phicoords0}) implies that there exist  $1 \leq s,t \leq k$ and $s \nsim t$ such that for any $x \in \Omega_s$ and $y \in \Omega_t, \ \big|u_{\sigma_i}^{\phi}(x)-u_{\sigma_i}^{\phi}(y) \big|$ is unbounded as $\sigma_i$ is small enough. Subsequently, $u_{\sigma_i}^{\phi}(x_k) \to \infty$ and $u_{\sigma_i}^{\phi}(x_1) \to -\infty$ as $i \to \infty,$ which results in
\begin{equation}\label{rho_sigma}
\rho_{\sigma_i}:=d_{\overline{\bf H}}(\gamma_{\sigma_i}(0), P_0) \to 0.
\end{equation}
Denote the intersection of $\Gamma_{\sigma_i}$ and the image of $L_{m,i}, m=2,...,k-1,$ as $\{P_2,...,P_{k-1}\}$ correspondingly. Then, $d_{\overline{\bf H}}(P_m,P_0)=P_m^{\rho}$ also converges to zero as $\sigma_i \to 0$. Thus, we define
\begin{equation}\label{l0}
    \Lambda_{\sigma_i}:= \Gamma_{\sigma_i} \cup \{(\rho,\phi): \rho \geq P_m^{\rho}, \phi=P_m^{\phi},m=2,...,k-1\}.
\end{equation}
As $\sigma_i \to 0, \ \Lambda_{\sigma_i}$ is as in the Figure 3 and 
\begin{equation}\label{l&L0}
   \sup_{x \in B_1(0)} d_{\overline{\bf H}}(\Lambda_{\sigma_i},\text{Im} \, L_i)  \to  0.
\end{equation}
Combining Lemma \ref{u&L0} and (\ref{l&L0}) together, we have the following approximation:

\begin{lemma} \label{linearapproxblowupsym0}
Let  $\{u_{\sigma_i}\}$ be the blow-up maps. Then,
\[
\lim_{\sigma_i \to 0} \sup_{x \in B_1(0)} d_{\overline{\bf H}}(u_{\sigma_i}(x),\Lambda_{\sigma_i}) = 0.
\]
\end{lemma}

Given arbitrary $\epsilon>0$, we define $R, r>0$ as follows:  
\begin{itemize}
\item Let $R \in (\frac{7}{8},1)$ such that 
\begin{equation} \label{ep0}
m(B_1(0) \backslash B_R(0)) < \frac{\epsilon}{2},
\end{equation}
where measure $m$ is induced from the domain metric $g$ in $B_1(0).$
\item Let  $r>0$ such that
\begin{equation} \label{fep0}
m(\{x \in B_R(0):  d_A(u_*(x),u_*(0))<2r\})<\frac{\epsilon}{2}.
\end{equation}
\end{itemize}

\begin{lemma} \label{fits0}
Let  $\{u_{\sigma_i}\}$ be the blow-up maps.   For  $R, r\in (0,1)$ as above, there exists $\overline{\sigma_1}>0$ such that
\[
u^{-1}_{\sigma_i}(B_r(P_0)) \cap B_R(0)  \subset \{x \in B_R(0):  d_A(u_*(0),u_*(x))<2r\},  \ \ \forall \sigma_i \in (0,\overline{\sigma_1}].
 \]
\end{lemma}
\begin{proof}
Assume on the contrary that $\sigma_i \rightarrow 0$ and, for each $i \in \N$, there exists
\[
x_i \in  \left( u^{-1}_{\sigma_i}(B_r(P_0))  \cap B_R(0) \right) \backslash \{x \in B_R(0):  d_A(u_*(x),u_*(0))<2r\}.
\]
Take a subsequence $\{x_{i_j}\}$ of $\{x_i\}$ such that $x_{i_j} \rightarrow x_* \in \overline{u_*^{-1}(B_r(P_0)) \cap B_R(0)}$ by compactness. Then,
\begin{eqnarray*}
r  \geq  d_A(u_*(x_*),u_*(0)) \geq 2r,
\end{eqnarray*}
which is a contradiction.
\end{proof}

\begin{lemma} \label{bigawayfromPnot0}
For $r >0$ as above, there exists $\overline{\sigma_2} \in (0,1)$ such that 
\[
d_{\overline{\bf H}}(\Gamma_{\rho_{\sigma_i}/2} \backslash B_r(P_0), \Gamma_{\rho_{\sigma_i}} ) >\frac{r}{2}, \ \ \ \forall \sigma_i \in (0,\overline{\sigma_2}].
\]
\end{lemma}
\begin{proof}
This follows from the fact that $\rho_{\sigma_i} \rightarrow 0$ as $\sigma_i \rightarrow 0$ and Lemma~\ref{cvxsets}.
\end{proof}

\begin{lemma} \label{s30}
For any $\epsilon>0$, let $R,r>0$ be as in (\ref{ep0}) and (\ref{fep0}). Then there exists $\overline{\sigma_3}>0$ such that $\forall \sigma_i \leq (0,\overline{\sigma_3}],$
\begin{equation} \label{firstone0}
\sup_{x \in B_R(0)} d_{\overline{\bf H}}(u_{\sigma_i}(x), {\bf H}[{\rho_{\sigma_i}}/2]) <\frac{r}{4}, 
\end{equation}
and
\begin{equation} \label{secondone0}
m\left( \{x \in B_1(0):  u_{\sigma_i} (x) \notin {\ms}[\rho_{\sigma_i}/2] \right\})<\epsilon. 
\end{equation}
\end{lemma}
\begin{proof}
Following Lemma~\ref{fits0} and Lemma~\ref{bigawayfromPnot0}, pick $\overline{\sigma_1}, \overline{\sigma_2}>0$ such that 
\begin{equation} \label{lem410}
 u^{-1}_{\sigma_i}(B_r(P_0))  \cap B_R(0)  \subset \{x \in B_R(0):  d_A(u_*(x),u_*(0))<2r\},  \ \ \forall {\sigma_i} \in (0,\overline{\sigma_1}],
\end{equation}
and
\begin{equation} \label{lem420}
d_{\overline{\bf H}}(\Gamma_{\rho_{\sigma_i}/2} \backslash B_r(P_0), \Gamma_{\rho_{\sigma_i}} ) >\frac{r}{2}, \ \ \ \forall {\sigma_i} \in (0,\overline{\sigma_2}].
\end{equation}
Following Lemma~\ref{linearapproxblowupsym0}, choose $\overline{\sigma_3}  \leq \min \{\overline{\sigma_1}, \overline{\sigma_2}\}$ such that 
\begin{equation} \label{lem430}
\sup_{x \in B_R(0)} d_{\overline{\bf H}}(u_{\sigma_i}(x), \Lambda_{\sigma_i}) <\frac{r}{4}, \ \ \ \forall \sigma_i \in (0, \overline{\sigma_3}].
\end{equation}
Since $\Lambda_{\sigma_i} \subset {\bf H}[{\rho_{\sigma_i}}/2]$, we have $ d_{\overline{\bf H}}(u_{\sigma_i}(x), {\bf H}[{\rho_{\sigma_i}}/2])  \leq d_{\overline{\bf H}}(u_{\sigma_i}(x), \Lambda_{\sigma_i})$ which combined with (\ref{lem430}) implies inequality (\ref{firstone0}).

Next, we prove that (\ref{lem430}) implies (\ref{secondone0}).
If $u_{\sigma_i}(x) \notin {\ms}[\rho_{\sigma_i}/2] \cup  B_r(P_0)$, then Lemmas ~\ref{distest} and ~\ref{bigawayfromPnot0} imply
\[
d_{\overline{\ms}}(u_{\sigma_i}(x), \Lambda_{\sigma_i}) = d_{\overline{\ms}}(u_{\sigma_i}(x), \Gamma_{\rho_{\sigma_i}}) \geq d_{\overline{\bf H}}(\Gamma_{{\rho_{\sigma_i}}/2} \backslash B_r(P_0), \Gamma_{\rho_{\sigma_i}} ) >\frac{r}{2}
\]
which in turn implies $x \in B_1(0) \backslash B_R(0)$ by (\ref{lem430}).  
 In other words,
 \begin{eqnarray*}
\left\{x \in B_1(0):  u_{\sigma_i}(x) \notin {\ms}[{\rho_{\sigma_i}}/2]\right\} & \subset&  u^{-1}_{\sigma_i}(B_r(P_0))   \cup (B_1(0) \backslash B_R(0))
\\
&= &
\left(   u^{-1}_{\sigma_i}(B_r(P_0))  \cap B_R(0) \right)  \cup (B_1(0) \backslash B_R(0))
\end{eqnarray*}
which, in light of  (\ref{ep0}), (\ref{fep0}) and Lemma \ref{fits0}, proves the assertion.
\end{proof}

\label{sec:c1c20}
Now we are ready to define constants $c_1, c_2 \text{ and } \sigma_0$ which will be fixed throughout:
\begin{itemize}
\item Let $c_1>0$ be a constant such that 
for any $t \in [\frac{5}{8},\frac{7}{8}]$, and any subharmonic function $f$ defined on $(B_t(0),g)$ w.r.t. Riemannian metric $g,$
\[
\sup_{B_{\frac{1}{2}}(0)} f \leq 
c_1\int_{B_t(0)} f (x) \, d\mbox{vol}_{g}(x).
\]
\item
Let $\mathcal H$ be the set of harmonic maps  $w:B_1(0) \rightarrow \overline{\ms}$ with $w(0)=P_0$, $Ord^w(0)=\alpha$ and $I^w(1)=1$.
Let 
\[
c_2:=\sup_{w \in \mathcal H} \,\left[ \left(2^{n-1}\int_{\partial B_\frac{1}{2} (0)} d^2(w,P_0) d\Sigma \right)^{-\frac{1}{2}} \right].
\]

\item 
Fix $\epsilon>0$ such that
\begin{equation} \label{choice0}
\frac{16}{3} c_1 \epsilon  < \frac{1}{2^2c_2^2}
\end{equation}
where $c_1, c_2$ are the constants defined above. 
\item
Let $R,r>0$ be chosen as in (\ref{ep0}), (\ref{fep0}) respectively. Let $\overline{\sigma_1}, \overline{\sigma_2}, \overline{\sigma_3}$ be as in Lemmas~\ref{fits0}, \ref{bigawayfromPnot0}, \ref{s30} respectively and choose 
\begin{equation} \label{choicesigmanot0}
\sigma_0 :=\min\{ \overline{\sigma_1},\overline{\sigma_2},\overline{\sigma_3}\}.
\end{equation}
\end{itemize}
Define
\begin{equation} \label{bus0}
u_k(x):=u_{\frac{\sigma_0}{2^{k}}}(x), \ \ \ \ \Lambda_k:=\Lambda_{\frac{\sigma_0}{2^{k}}},
\end{equation}where $\Lambda_{\sigma}$ is defined as in (\ref{l0}). In particular,  for $k=0$, 
\[
u_0(x)=u_{\frac{\sigma_0}{2^0}}(x) = \lambda^u\left(\frac{\sigma_0}{2^0}\right) u\left(\frac{\sigma_0 x}{2^0}\right)=\lambda^u(\sigma_0)u(\sigma_0 x).
\]
We claim that for $k=1, 2, \dots$,
\begin{equation} \label{alt0}
u_k(x)= \lambda_{k-1}u_{k-1}\left(\frac{x}{2}\right), \ \ \ \ \ \lambda_{k-1}:=\left(2^{n-1}\displaystyle{\int_{\partial B_\frac{1}{2}(0)} d^2(u_{k-1},u_{k-1}(0)) d\Sigma}
\right)^{-\frac{1}{2}}.
\end{equation}
Indeed, assuming (\ref{alt0}) holds for $k=1, \dots, j-1$,  we have
\begin{eqnarray*}
\lambda_{k-1}u_{k-1}(\frac{x}{2}) &=& \lambda_{k-1} u_{\frac{\sigma_0}{2^{k-1}}}(\frac{x}{2}) = \lambda_{k-1} \lambda^u(\frac{\sigma_0}{2^{k-1}}) u(\frac{\sigma_0x}{2^k}),\\
u_k(x) & = & u_{\frac{\sigma_0}{2^{k}}}(x)=\lambda^u(\frac{\sigma_0}{2^k})u(\frac{\sigma_0x}{2^k}).
\end{eqnarray*}
Note that $\lambda_{k-1}\lambda^{u}(\frac{\sigma_0}{2^{k-1}})=\lambda^u(\frac{\sigma_0}{2^k})$ by an obvious calculation:
\begin{small}
\begin{eqnarray*}
    \lambda_{k-1}\lambda^{u}(\frac{\sigma_0}{2^{k-1}}) & = & \left(2^{n-1} \int_{\partial B_{\frac{1}{2}}(0)} d^2(u_{k-1},u_{k-1}(0)) \, d\Sigma \right)^{-\frac{1}{2}} \cdot \left( \left( \frac{\sigma_0}{2^{k-1}}\right)^{1-n} \int_{\partial B_{\frac{\sigma_0}{2^{k-1}}}} d^2(u,u(0)) \, d\Sigma\right)^{-\frac{1}{2}} \\
    & = & (2^{n-1})^{-\frac{1}{2}} \left( \int_{\partial B_{\frac{1}{2}}(0)} d^2(u(\frac{\sigma_0}{2^{k-1}}x),u(0)) \, d\Sigma \right)^{-\frac{1}{2}} \cdot 
    \left(\left(\frac{\sigma_0}{2^{k-1}}\right)^{1-n} \left( \int_{\partial B_{\frac{\sigma_0}{2^{k-1}}}(0)} d^2(u, u(0)) \, d\Sigma \right) \right)^{\frac{1}{2}} \\ 
    & \cdot & \left( \left( \frac{\sigma_0}{2^{k-1}}\right)^{1-n} \int_{\partial B_{\frac{\sigma_0}{2^{k-1}}}} d^2(u,u(0)) \, d\Sigma\right)^{-\frac{1}{2}}  \\
    & = & (2^{n-1})^{-\frac{1}{2}} \left( \int_{\partial B_{\frac{1}{2}}(0)} d^2(u(\frac{\sigma_0}{2^{k-1}}x),u(0)) \, d\Sigma \right)^{-\frac{1}{2}}\\
    & = & \left( \frac{1}{2}\right)^{\frac{n-1}{2}} \cdot \left( \frac{\sigma_0}{2^{k-1}}\right)^{\frac{n-1}{2}} \left( \int_{\partial B_{\frac{\sigma_0}{2^k}}(0)} d^2(u,u(0)) \, d\Sigma \right)^{-\frac{1}{2}}\\
    & =& \left( \frac{\sigma_0}{2^k} \right)^{\frac{n-1}{2}}\left( \int_{\partial B_{\frac{\sigma_0}{2^k}}(0)} d^2(u,u(0)) \, d\Sigma \right)^{-\frac{1}{2}}\\
    & = & \left( \left( \frac{\sigma_0}{2^k}\right)^{1-n}\int_{\partial B_{\frac{\sigma_0}{2^k}}(0)} d^2(u,u(0)) \, d\Sigma\right)^{-\frac{1}{2}}\\
    & = &\lambda^u(\frac{\sigma_0}{2^k}).
\end{eqnarray*}
\end{small}

\begin{proof}[Proof of Theorem \ref{result0}]
We assume on the contrary that $u(x_0)=P_0$.  Let $\{u_\sigma\}$ be the blow-up maps at $x_0=0$. Let $\epsilon, \sigma_0, R, r$ be as in (\ref{choice0}), (\ref{choicesigmanot0}), (\ref{ep0}), (\ref{fep0}) respectively   to define the sequence of maps $\{u_k\}_{k=0}^{\infty}$ as in (\ref{bus0}). 
We claim  
\begin{equation} \label{indeq0}
\sup_{x \in B_R(0)} d_{\overline{\bf H}}(u_k(x), {\bf H}[\rho_0/2])
<\frac{r}{2^{k+2}}, \ \ \ \forall k=0,1,2,\dots,
\end{equation}
where $\rho_0:=d_{\overline{\bf H}}(\gamma_{\sigma_0}(0),P_0)>0$ (cf. (\ref{rho_sigma})).
To prove (\ref{indeq0}), 
first notice that $\sigma_0 \leq \overline{\sigma_3}$ implies (cf.~Lemma~\ref{s30})
\[
\sup_{x \in B_R(0)} d_{\overline{\bf H}}(u_0(x), {\bf H}[\rho_0/2])
<\frac{r}{4}.
\]
We now proceed by induction.
Assume
\begin{eqnarray*}
\sup_{x \in B_R(0)} d_{\overline{\bf H}}(u_{k-1}(x), {\bf H}[\rho_0/2])
<\frac{r}{2^{k+1}}.
\end{eqnarray*}
Since $\frac{\sigma_0}{2^{k-1}} \leq \overline{\sigma_3}$, Lemma~\ref{s30} and Fubini theorem imply that 
\begin{align*}
&\mathop{\min_{\frac{5}{8} \leq \tau \leq \frac{7}{8}}} m\left(\{x \in \partial B_{\tau}(0): u_{k-1} \notin {\bf H}[\rho_0/2]\}\right) \cdot \frac{3}{16} \\ \leq &\int_{5/8}^{7/8} m\left(\{x \in \partial B_{\tau}(0): u_{k-1}(x) \notin {\bf H}[\rho_0/2]\}\right) \tau \, d\tau\\ 
= & \ m\left( \{ x \in B_{\frac{7}{8}}(0) \setminus B_{\frac{5}{8}}(0): u_{k-1}(x) \notin {\bf H}[\rho_0/2]\}\right) 
<  \epsilon,
\end{align*}
which indicates that there exists $\tau_0 \in [\frac{5}{8},\frac{7}{8}]$ such that 
\[
m\left( \{x \in \partial B_{\tau_0}(0):  u_{k-1} (x) \notin {\ms}[\rho_0/2] \}\right)<\frac{16}{3} \epsilon . 
\]
Let $h:B_{\tau_0}(0) \rightarrow \overline{\bf H}$ be a harmonic map with boundary values $\pi \circ u_{k-1}|_{\partial B_{\tau_0}(0)}$ where $\pi: \overline{\bf H} \rightarrow {\bf H}[\rho_0/2]$ is the nearest point projection map.  We therefore have the following dichotomy for $x \in \partial B_{\tau_0}(0)$:  either (i) $u_{k-1}(x) =h(x)$ or (ii) $u_{k-1}(x) \neq h(x)$ and $d_{\overline{\bf H}}(u_{k-1}(x),h(x)) < \frac{r}{2^{k+1}}$.  Since $d_{\overline{\bf H}}^2(u,h)$ is a subharmonic, we have 
\begin{eqnarray*}
\sup_{x \in B_\frac{1}{2}(0)} d^2_{\bf H}(u_{k-1}(x), {\bf H}[\rho_0/2]) 
& \leq & 
\sup_{x \in B_{\frac{1}{2}}(0)} d_{\overline{\bf H}}^2(u_{k-1}(x),h(x)) 
\\
& \leq & c_1 \int_{\partial B_{\tau_0}(0)} d_{\overline{\bf H}}^2(u_{k-1},h) d\Sigma
\\
&  \leq & \frac{16}{3} c_1 \epsilon \frac{r^2}{2^{2(k+1)}} <\frac{r^2}{2^{2(k+2)}c_2^2}.
\end{eqnarray*}
In other words, 
\[
\sup_{x \in B_\frac{1}{2}(0)} d_{\overline{\bf H}}(u_{k-1}(x), {\bf H}[\rho_0/2]) < \frac{r}{2^{k+2}c_2}.
\]
Multiplying both sides of the inequality by $\lambda_{k-1}$ and noting (\ref{alt0}),  we obtain
\[
\sup_{x \in B_1(0)} d_{\overline{\bf H}}(u_k(x), \lambda_{k-1} {\bf H}[\rho_0/2]) <  \frac{\lambda_{k-1} r}{2^{k+2}c_2} \leq \frac{r}{2^{k+2}}
\]
Since $I^{u_{k}}(1)=1$ holds for any $k$, then $1 \leq \lambda_{k-1}=\lambda^{u_{k-1}}(\frac{1}{2})$ (cf.~(\ref{lowerbdlambda0})).   Thus, by Lemma~\ref{geomH},
\[
\lambda_{k-1} {\bf H}[\rho_0/2] =  {\bf H}[\lambda_{k-1}\rho_0/2] \subseteq {\bf H}[\rho_0/2] .
\]
Combining the above yields (\ref{indeq0}).

Finally, since
\[
\frac{\rho_0}{2} =d_{\overline{\bf H}}(u_k(0), {\bf H}[\rho_0/2]) \leq \sup_{x \in B_1(0)} d_{\overline{\bf H}}(u_k(x), {\bf H}[\rho_0/2]) < \frac{r}{2^{k+2}} .
\]
We get a contradiction for $k$ large enough.
\end{proof}

\section{Harmonic Map into $\overline{\mathcal{T}}$} \label{sec:blowup}

In this section, we first prove Theorem \ref{nosing} and then apply it to show Theorem \ref{mainresult}. Let $u:\Omega \to \overline{\mathcal{T}}$ be a harmonic map so that $u(\Omega) \cap \mathcal{T} \neq \emptyset$ i.e. $u(\Omega) \not\subset \partial \mathcal{T}.$ Define {\it singular set} as 
\[\mathcal{S}(u)=\{x \in \Omega : u(x) \in \partial \mathcal{T}\}.
\]
Theorem \ref{cont} implies that $u$ is continuous and thus $\mathcal{S}(u)$ is a closed set.

Assume that $\mathcal{S}(u) \neq \emptyset.$ We can decompose $\mathcal{S}(u)$ as
\[
\mathcal{S}(u)=\bigcup \mathcal{S}_j(u),
\]
where $\mathcal{S}_j(u)$ consists of singular points $x \in \Omega$ such that $u(x) \in \partial \mathcal{T}$ is contained in the $j$-dimensional open stratum $\mathcal{T}'$. Given $x_0 \in \partial\mathcal{S}(u) \cap \mathcal{S}_j(u)$ and $Ord^u(x_0)=\alpha >0,$ let $r_0>0$ such that $B_{r_0}(x_0) \subset \Omega.$ Identify $(B_{r_0}(x_0),g)$ with Euclidean ball $B_{r_0}(0)$ and $x_0=0$ via normal coordinates. By the stratification preserving homeomorphism (\ref{hom}), let
\begin{equation}\label{decomp}
u=(V,v)=(V,v^1,...,v^{k-j}): (B_{r_0}(0),g) \to \mathcal{U} \times \mathcal{V} \subset \mathbb{C}^j \times \overline{\bf{H}}^{k-j}
\end{equation}
be a local representation with $V(0)=0 \text{ and } v^{\eta}(0)=P_0$ for each $\eta \in \{1,...,k-j\}. $ 

We claim that each component $v^{\eta}$ is non-constant. To see this, we construct sequence $\{x_i\}\subset B_{r_0}(0)$ such that (i) $\epsilon_i \to 0$ as $i \to +\infty$ and (ii) for each $i \in \mathbb{N}, \ x_i \in B_{\epsilon_i}(0)\cap \mathcal{S}(u)^c \subset B_{r_0}(0).$ This results in $x_i \to x_0=0$ and $u(x_i) \in \mathcal{T}$ for each $i \in \mathbb{N}.$ Hence $v^{\eta}(x_i) \neq P_0$ for each $\eta \in \{1,...,k-j\}$ and $i \in \mathbb{N},$ leading to that singular components $v^{\eta}: (B_{r_0}(0),g) \to \overline{\bf H}$ of $v$ are non-constant. 

We define a function $\lambda^u: (0,r_0] \rightarrow (0,\infty)$ by
\[
\lambda^u(\sigma) = \left(\sigma^{1-n} \int_{\partial B_\sigma(0)} d^2(u,u(0)) d\Sigma \right)^{-\frac{1}{2}}.
\]
For $\sigma \in (0,r_0]$, the {\it blow-up map} of $u$ at $x_0$ is given by
\begin{eqnarray*}
u_{\sigma}:(B_1(0),g) \rightarrow \mathcal{U} \times \mathcal{V}, \ \ \ u_{\sigma}(x)=\lambda^u(\sigma) u(\sigma x) &=& (\lambda^{u}(\sigma)V(\sigma x),\lambda^{u}(\sigma)v(\sigma x))\\
&=&(V_{\sigma}(x),v_{\sigma}(x))\\
&=& (V_{\sigma}(x),v^1_{\sigma}(x),...,v^{k-j}_{\sigma}(x)),
\end{eqnarray*}
where $v_{\sigma}^{\eta}(0)=P_0$ for $\eta=1,...,k-j.$ As $\sigma \to 0,$ we show in Appendix II that there exists a subsequence $\{v_{\sigma_i}\}$ converging locally uniformly in the pullback sense (cf. section \ref{pullback}) to a homogeneous harmonic map 
\[v_*=(v_*^{1},...,v_*^{k-j}):(B_1(0),g) \to (\overline{\bf H}_*^{k-j},d)= (\overline{\bf H}_* \times ... \times \overline{\bf H}_*,d),\]
where $\overline{\bf H}_*$ is an abstract NPC space. In other words, for $\eta \in \{1,...,k-j\},$
\begin{equation}\label{distapprox1}
d_{\overline{\bf H}}(v_{\sigma_i}^{\eta}(\cdot),v_{\sigma_i}^{\eta}(\cdot)) \to d(v_*^{\eta}(\cdot),v_*^{\eta}(\cdot)) \mbox{ uniformly on compact subsets of } B_1(0).
\end{equation} 

The main difference between this section \ref{sec:blowup} and section \ref{sec:blowup0} is that $V \text{ and } v$ are not harmonic maps because the WP-metric is only {\it asymptotically} a product metric near $\partial \mathcal{T}$ from \cite{6}, which implies that the harmonic map equation doesn't hold for $V \text{ and } v.$ Following \cite[Lemma 4.19]{1}, $\{v_{\sigma_i}\}$ is called a sequence of {\it asymptotically harmonic maps.} 
\begin{definition}\label{asy}
    A sequence of maps $v_{\sigma_i}:(B_1(0),g_i) \to \overline{\bf H}^{k-j}$ with $v_{\sigma_i}(0)=P_0$ and $g_i(x)=g(\sigma_i x)$ is a {\it sequence of asymptotically harmonic maps} if the following conditions are fulfilled:
    \begin{itemize}
        \item [(i)] The sequence of metrics $g_i$ on $B_1(0) \subset \mathbf{R}^n$ converges to the Euclidean metric in $C^{\infty}.$
        \item [(ii)] There exists a constant $E_0 >0$ such that $E^{v_{\sigma_i}}(\vartheta) \leq {\vartheta}^{n}E_0 \text{ for every } \vartheta \in (0,\frac{3}{4}]$ where $n$ is the dimension of  $B_1(0).$
        \item [(iii)] $v_{\sigma_i}$ converges locally uniformly in the pullback sense to a homogeneous harmonic map $v_*:B_1(0) \to (\overline{\bf H}^{k-j}_*,d)$ into an NPC space.
        \item[(iv)] For any fixed $R \in (0,1), r \in (0,1) \text{ and } d >0,$ there exists $c_0 >0$ derived from \cite[Lemma 4.18]{1} such that for any harmonic map $w:(B_R(0), g_i) \to \overline{\bf H}^{k-j}$ with
        \[
        \sup_{B_R(0)} d_{\overline{\bf H}}(w, P_0 ) \leq d,
        \]
        we have 
        \[
        \sup_{B_{r \vartheta}(0)} d_{\overline{\bf H}}^2(v_{\sigma_i},w) \leq \frac{c_0}{\vartheta^{n-1}} \int_{\partial{B_{\vartheta}(0)}} d_{\overline{\bf H}}^2(v_{\sigma_i},w) \, d\Sigma_{g_i} +c_0 \sigma_i^2 \vartheta^3, \, \forall \vartheta \in (0,R]
        \]
        where $\Sigma_{g_i}$ is the volume form on $\partial B_{\vartheta}(0)$ with respect to the metric $g_i$.
    \end{itemize}
\end{definition}

For the proof of Theorem \ref{nosing}, we consider two cases: (i) $v_*:B_1(0) \to (\overline{\bf H}^{k-j}_*,d)$ is non-constant and (ii) $v_*$ is a constant map.

\subsection{Case I: Non-constant Pullback Limit $v_*$}\label{sec:blowup1}
This section focuses on the case that there exists a non-constant component map $v^{\eta_0}_*: B_1(0) \to (\overline{\bf H}_*,d)$ derived from $v^{\eta_0}$ for some $\eta_0 \in \{1,...,k-j\}.$ For an abuse of notation, we denote $v^{\eta_0}_{\sigma_i} \text{ and } v^{\eta_0}_*$ by $v_{\sigma_i} \text{ and }v_*.$ 

\begin{remark}
     The nonconstant homogeneous harmonic $v_*$ is {\it piecewise a function} in the meaning of Definition \ref{pw0} into the metric space $\overline{\bf H}_A,$ where $A$ is defined as in section \ref{sec:blowup0} with $u_*$ replaced by $v_*$ (cf. (\ref{A})). For the sake of completeness, we provide the proof of this fact in Appendix II. Thus, we rewrite the distance convergence (\ref{distapprox1}) by
     \begin{equation}\label{distapprox1.1}
         d_{\overline{\bf H}}(v_{\sigma_i}(\cdot),v_{\sigma_i}(\cdot)) \to d_A(v_*(\cdot),v_*(\cdot)) \mbox{ uniformly on compact subsets of } B_1(0).
    \end{equation}
    We then define {\it normalized blow-up maps} $v_{\sigma_i}$ following the idea of Definition \ref{normalized} and construct the corresponding sequence $\{L_i\}$ (cf. (\ref{L_i})) by replacing $u_{\sigma_i}, u_*$ with $v_{\sigma_i},v_*$ respectively. From Remark \ref{2pts}, the properties of $v_*$ and (\ref{distapprox1.1}) ensure that Lemmas \ref{approx0} -- \ref{s30} also hold for normalized blow-up maps $v_{\sigma_i}.$
\end{remark}

The constant $c$ and $\epsilon>0$ defined below will be fixed throughout.
\begin{itemize}
\item
Let $\mathcal H$ be the set of harmonic maps $w:B_1(0) \rightarrow \mathbb{C}^j \times \overline{\ms}^{k-j}$ with $w(0)=(0,P_0)$, $Ord^w(0)=\alpha$, $I^w(1)=1$ and $E^w(1) \leq 2\alpha.$ Let 
\begin{equation}\label{c}
c:=\sup_{w \in \mathcal H} \,\left[ \left(2^{n-1}\int_{\partial B_\frac{1}{2} (0)} d^2(w,w(0)) d\Sigma \right)^{-\frac{1}{2}} \right].
\end{equation}
\item
Fix $\epsilon>0$ in Lemma \ref{s30} throughout such that
\begin{equation} \label{choice}
\frac{16}{3} \epsilon  < \frac{1}{2^2c^2}
\end{equation}
where $c$ is the constant defined above.
\item
Let $R>0$ be as in (\ref{ep0}) and $r>0$ be as in (\ref{fep0}) with $u_*$ replaced by $v_*$. Let $\overline{\sigma_1}, \overline{\sigma_2}, \overline{\sigma_3}$ be as in Lemmas~\ref{fits0}, \ref{bigawayfromPnot0}, \ref{s30} respectively with respect to $v_{\sigma_i}$ and $v_*.$ Set  
\begin{equation} \label{choicesigmanot}
    \sigma_0 :=\min\{ \overline{\sigma_1},\overline{\sigma_2},\overline{\sigma_3}\} 
\end{equation}
satisfying that $\sigma_0$ is sufficiently small such that $c_0 \left(\frac{\sigma_0}{2^k}\right)^2\leq \frac{r^2}{2^{2(k+2)}}\frac{16}{3}\epsilon$ holds for any $k \in \mathbb{N},$ where $c_0$ is the constant as in Definition \ref{asy}(iv). 
\end{itemize}
\begin{remark}

    The constant $c$ is bounded away from zero. This follows from the monotonicity formula (\ref{monotonicityformula0}):
    \[
    \frac{I^w(\frac{1}{2})}{(\frac{1}{2})^{n+1}}=2^2(\lambda^w(\frac{1}{2}))^{-2}\leq I^w(1)=1,
    \]
    which implies that
    \[
    1<\lambda^w(\frac{1}{2})=\left(2^{n-1}\int_{\partial B_{\frac{1}{2}}(0)} d^2(w,w(0)) \,d\Sigma\right)^{-\frac{1}{2}}.
    \]
\end{remark}
Define
\begin{align}
    u_k(x)&:=u_{\frac{\sigma_0}{2^k}}(x)=(V_{\frac{\sigma_0}{2^k}}(x),v_{\frac{\sigma_0}{2^k}}(x)),\\
    v_k(x)&:=v_{\frac{\sigma_0}{2^{k}}}(x)=\lambda^{u}\left(\frac{\sigma_0}{2^k}\right)v^{\eta_0}\left(\frac{\sigma_0}{2^k}x\right)=\left(\left({\frac{\sigma_0}{2^k}}\right)^{1-n}\displaystyle{\int_{\partial B_{\frac{\sigma_0}{2^k}}(0)}d^2(u,u(0)) d\Sigma} \right)^{-\frac{1}{2}}v^{\eta_0}\left(\frac{\sigma_0}{2^k}x\right).\label{bus}
\end{align}
In particular,  for $k=0$, 
\[
v_0(x)=v_{\frac{\sigma_0}{2^0}}(x) = \lambda^u\left(\frac{\sigma_0}{2^0}\right) v^{\eta_0}\left(\frac{\sigma_0 x}{2^0}\right)=\lambda^u(\sigma_0)v^{\eta_0}(\sigma_0 x).
\]
We claim that for $k=1, 2, \dots$,
\begin{equation} \label{alt}
v_k(x)= \lambda_{k-1}v_{k-1}\left(\frac{x}{2}\right), \mbox{ where } \lambda_{k-1}:=\left(2^{n-1}\displaystyle{\int_{\partial B_\frac{1}{2}(0)} d^2(u_{k-1},u_{k-1}(0)) d\Sigma}
\right)^{-\frac{1}{2}}.
\end{equation}
Indeed, assuming (\ref{alt}) holds for $k=1, \dots, j-1$,  we have
\begin{eqnarray*}
\lambda_{k-1}v_{k-1}\left(\frac{x}{2}\right) &=& \lambda_{k-1} v_{\frac{\sigma_0}{2^{k-1}}}\left(\frac{x}{2}\right) = \lambda_{k-1} \lambda^u\left(\frac{\sigma_0}{2^{k-1}}\right) v^{\eta_0}\left(\frac{\sigma_0x}{2^k}\right),\\
v_k(x) & = & v_{\frac{\sigma_0}{2^{k}}}(x)=\lambda^u\left(\frac{\sigma_0}{2^k}\right)v^{\eta_0}\left(\frac{\sigma_0x}{2^k}\right).
\end{eqnarray*}
Since we have computed $\lambda_{k-1}\lambda^{u}(\frac{\sigma_0}{2^{k-1}})=\lambda^u(\frac{\sigma_0}{2^k})$ in section \ref{sec:blowup0}, 
then $\lambda_{k-1}v_{k-1}(\frac{x}{2})=v_k(x).$

\begin{proof}[Proof of Theorem \ref{nosing} for Case I]
From the decomposition (\ref{decomp}) near the chosen singular point $x_0,$ we assume $v^{\eta_0}(x_0)=v^{\eta_0}(0)=P_0$.  Let $\{v_\sigma\}$ be the blow-up maps at $x_0=0$. Let $\epsilon, \sigma_0, R, r$ be as in (\ref{choice}), (\ref{choicesigmanot}), (\ref{ep0}), and (\ref{fep0}) respectively  to define the sequence of maps $\{v_k\}_{k=0}^{\infty}$ as in (\ref{bus}). 
We claim  
\begin{equation} \label{indeq}
\sup_{x \in B_R(0)} d_{\overline{\bf H}}(v_k(x), {\bf H}[\rho_0/2])
<\frac{r}{2^{k+2}}, \ \ \ \forall k=0,1,2,\dots,
\end{equation}
where $\rho_0:=d_{\overline{\bf H}}(P_0,\gamma_{\sigma_0}(0))$ (cf. (\ref{l0})).
To prove (\ref{indeq}), 
firstly $\sigma_0 \leq \overline{\sigma_3}$ implies (cf.~Lemma~\ref{s30})
\[
\sup_{x \in B_R(0)} d_{\overline{\bf H}}(v_0(x), {\bf H}[\rho_0/2])
<\frac{r}{4}.
\]
We now proceed by induction. Assume
\begin{eqnarray*}
\sup_{x \in B_R(0)} d_{\overline{\bf H}}(v_{k-1}(x), {\bf H}[\rho_0/2])
<\frac{r}{2^{k+1}}.
\end{eqnarray*}
Since $\frac{\sigma_0}{2^{k-1}} \leq \overline{\sigma_3}$, Lemma~\ref{s30} and Fubini's theorem imply that there exists $\tau \in [\frac{5}{8},\frac{7}{8}]$ such that 
\[
m\left( \{x \in \partial B_\tau(0):  v_{k-1} (x) \notin {\ms}[\rho_0/2] \}\right)<\frac{16}{3} \epsilon . 
\]
Let $w:B_\tau(0) \rightarrow \overline{\bf H}$ be a harmonic map with boundary values $\pi \circ v_{k-1}|_{\partial B_\tau(0)}$ where $\pi: \overline{\bf H} \rightarrow {\bf H}[\rho_0/2]$ is the nearest point projection map. Therefore, for $x \in \partial B_\tau(0),$ either (i) $v_{k-1}(x) =w(x)$ or (ii) $v_{k-1}(x) \neq w(x)$ and $d_{\overline{\bf H}}(v_{k-1}(x),w(x)) < \frac{r}{2^{k+1}}$. From Definition \ref{asy} and (\ref{choicesigmanot}), we fix
\[
\vartheta=\tau \in \left[\frac{5}{8},\frac{7}{8}\right], \ r \in(0,1) \text{ such that } r \vartheta=\frac{1}{2},
\]
then there exists constant $c_0 >0$ and sequence $\{c_{k-1}:=c_0 \left(\frac{\sigma_0}{2^{k-1}}\right)^2\}$ such that $c_{k-1} \leq \frac{r^2}{2^{2(k+1)}}\frac{16}{3}\epsilon$ for any $k \in \mathbb{N}$ (cf. (\ref{choicesigmanot})),
\begin{eqnarray*}
\sup_{x \in B_\frac{1}{2}(0)} d^2_{\overline{\bf H}}(v_{k-1}(x), {\bf H}[\rho_0/2]) 
& \leq & 
\sup_{x \in B_{\frac{1}{2}}(0)}d_{\overline{\bf H}}^2(v_{k-1}(x),w(x))
\\
& \leq & \frac{c_0}{\tau^{n-1}} \int_{\partial B_{\tau}(0)} d^2_{\overline{\bf H}}(v_{k-1}(x),w(x)) \, d\Sigma +c_{k-1}\tau^3
\\
& \leq & \frac{c_0}{\tau^{n-1}} \frac{16}{3} \epsilon \frac{r^2}{2^{2(k+1)}} + c_{k-1}\tau^3
\\
& \leq & \frac{r^2}{2^{2(k+1)}}\frac{16}{3}\epsilon \, \left(c_0   \left(\frac{8}{5}\right)^{n-1} + 1 \right)
\\
& < & A \frac{r^2}{2^{2(k+2)}c^2},
\end{eqnarray*}
where $A$ is a constant.
In other words,
\[
\sup_{x \in B_\frac{1}{2}(0)} d_{\overline{\bf H}}(v_{k-1}(x), {\bf H}[\rho_0/2]) < \sqrt{A} \frac{r}{2^{k+2}c}.
\]
Multiplying both sides of the inequality by $\lambda_{k-1}$ and noting (\ref{c}) and (\ref{alt}),  we obtain
\[
\sup_{x \in B_1(0)} d_{\overline{\bf H}}(v_k(x), \lambda_{k-1} {\bf H}[\rho_0/2]) < \sqrt{A} \frac{\lambda_{k-1} r}{2^{k+2}c} \leq \sqrt{A} \frac{r}{2^{k+2}}.
\]
Since $I^{u_{k}}(1)=1$ holds for any $k$, then $1 \leq \lambda_{k-1}=\lambda^{u_{k-1}}(\frac{1}{2})$ (cf.~(\ref{lowerbdlambda0})).   Thus, by Lemma~\ref{geomH},
\[
\lambda_{k-1} {\bf H}[\rho_0/2] =  {\bf H}[\lambda_{k-1}\rho_0/2] \subseteq {\bf H}[\rho_0/2] .
\]
Combining the above two equations yields (\ref{indeq}).

Finally, (\ref{indeq}) implies
\[
\frac{\rho_0}{2} =d_{\overline{\bf H}}(v_k(0), {\bf H}[\rho_0/2]) \leq \sup_{x \in B_1(0)} d_{\overline{\bf H}}(v_k(x), {\bf H}[\rho_0/2]) < \sqrt{A} \frac{r}{2^{k+2}} .
\]
This is a contradiction for $k$ large. Thus, $v^{\eta_0}(0) \neq P_0$ for some $\eta_0,$ which contradicts the assumption that $v^{\eta}(0)=P_0$ for all $1\leq \eta \leq k-j$ according to (\ref{decomp}). Consequently, $\mathcal{S}(u)=\emptyset.$
\end{proof}

\subsection{Case II: Constant Pullback Limit $v_*$}\label{sec:blowup2}

This section deals with the case that $v_*$ is a constant pullback limit map, which means that the component map $v_*^{\eta}$ is constant for any $\eta.$ In order to guarantee that the pullback limit of the sequence of blow-up maps derived from $v: B_{r_0}(0) \to \overline{\bf H}^{k-j}$ is non-constant, we define another modification factor $\lambda^v:(0,r_0] \to (0,\infty)$ by
\[
\lambda^v(\sigma) = \left(\sigma^{1-n} \int_{\partial B_\sigma(0)} d^2(v,v(0)) d\Sigma \right)^{-\frac{1}{2}}.
\]
For $\sigma \in (0,r_0],$ the alternative blow-up map of $v$ at $x_0=0$ is given by
\[
\tilde{v}_{\sigma}(x):=\lambda^v(\sigma) v(\sigma x)=(\lambda^v(\sigma)v^1(\sigma x),...,\lambda^v(\sigma) v^{k-j}(\sigma x)): (B_1(0),g) \to \overline{\bf H}^{k-j},
\]
where $\tilde{v}_{\sigma}^{\eta}(0)=P_0$ for each $\eta \in \{1,...,k-j\}.$ As $\sigma$ tends to zero, \cite[Lemma 4.30]{1} asserts that there exists a subsequence $\{\tilde{v}_{\sigma_i}\}$ of alternative blow-up maps converging locally uniformly in the pullback sense to a homogeneous harmonic map $\tilde{v}_*: (B_1(0),g) \to (\overline{\bf H}^{k-j}_*,d)$ such that

\begin{equation}\label{distapprox2}
    d_{\overline{\bf H}}(\tilde{v}_{\sigma_i}(\cdot),\tilde{v}_{\sigma_i}(\cdot)) \to d(\tilde{v}_*(\cdot),\tilde{v}_*(\cdot)) \mbox{ uniformly on compact subsets of } B_1(0). 
\end{equation}
From \cite[Lemma 4.32]{1}, $\{\tilde{v}_{\sigma_i}: (B_1(0),g) \to \overline{\bf H}^{k-j}\}$ is a sequence of asymptotically harmonic maps with $\tilde{v}_{\sigma_i}^{\eta}(0)=P_0$ where $\eta=1,...,k-j.$ By \cite[Lemma 49]{8}, the limit map $\tilde{v}_*$ is non-constant i.e. the component $\tilde{v}_*^{\eta_0}$ is non-constant for some $\eta_0 \in \{1,...,k-j\}.$ For simplicity, denote $\tilde{v}_{\sigma}^{\eta_0}=\lambda^v(\sigma)v^{\eta_0}(\sigma x)$ and $\tilde{v}_*^{\eta_0}$ by $\hat{v}_{\sigma}: (B_1(0),g) \to \overline{\bf H}$ and $\hat{v}_*:(B_1(0),g) \to (\overline{\bf H}_*,d)$ respectively.

The monotonicity of $v: (B_{r_0}(0),g) \to \overline{\bf H}^{k-j}$ is introduced in \cite[Proposition 4.24]{1}: The order of $v$ is well-defined at any singular point $x_0\in \mathcal{S}(u)$ given by
\begin{equation}
Ord^v(x_0):=\lim_{r \to 0} \frac{r E^v(r)}{I^v(r)} =\beta >0,
\end{equation}
where
\begin{eqnarray*}
E^v(r) :=\int_{B_{r}(0)} |\nabla v|^2 d\mu \ \mbox{ and } \
I^v(r) := \int_{\partial B_{r}(0)} d^2(v(x),v(0)) d\Sigma.
\end{eqnarray*}
There exist $C >0$ and $R_0 >0$ depending continuously on the point $x_0$ such that
\begin{equation} \label{monotonicityformula}
r \mapsto e^{C r}\frac{E^{v}(r) }{r^{n-2+2\beta}}, \ \ \  \ \ r \mapsto e^{C r}\frac{I^{v}(r) }{r^{n-1+2\beta}}, \ \ \ 
 \end{equation}
are non-decreasing functions for $r \in (0, R_0).$ Monotonicity property (\ref{monotonicityformula}) implies
\[
e^{\frac{C}{2}}\frac{I^{\tilde{v}_{\sigma}}(\frac{1}{2})}{{(\frac{1}{2})}^{n-1+2\beta}}=e^{\frac{C}{2}}(\lambda^{\tilde{v}_{\sigma}}(\frac{1}{2}))^{-2}4^{\beta} \leq e^{C} \, I^{\tilde{v}_{\sigma}}(1)=e^{C}.
\]
For domain metric $g$ sufficiently close to the Euclidean metric on $B_1(0)$, i.e. for $C$ close to 0,
\begin{equation}\label{lowerbdlambda2}
1 < 2^{\beta} \leq \lambda^{\tilde{v}_{\sigma}}(\frac{1}{2}).
\end{equation}

\begin{remark}
The non-constant homogeneous harmonic map $\hat{v}_*$ is {\it piecewise a function} (cf. Definition \ref{pw0}) into $\overline{\bf H}_A$ defined in section \ref{2.3} and (\ref{A}) with $u_*$ replaced by $\hat{v}_*$. For the convenience of the reader, this fact is shown in Appendix III. Analogous to the arguments of sections \ref{sec:blowup0} and \ref{sec:blowup1}, we derive
\begin{equation}\label{distapprox2.1}
    d_{\overline{\bf H}}(\hat{v}_{\sigma_i}(\cdot),\hat{v}_{\sigma_i}(\cdot)) \to d_A(\hat{v}_*(\cdot),\hat{v}_*(\cdot)) \mbox{ uniformly on compact subsets of } B_1(0)
\end{equation}
from (\ref{distapprox2}) and construct {\it alternative normalized maps} $\hat{v}_{\sigma_i}$ according to Definition \ref{normalized} and $\{L_i\}$ in the context of (\ref{L_i}) by replacing $u_{\sigma_i},u_*$ by $\hat{v}_{\sigma_i}$ and $\hat{v}_*.$
These facts of $\hat{v}_*$ and (\ref{distapprox2.1}) guarantee that Lemmas \ref{approx0} -- \ref{s30} remains valid with $u_{\sigma_i}$ substituted by alternative normalized maps $\hat{v}_{\sigma_i}.$
\end{remark}

Let $R>0$ be as in (\ref{ep0}) and $r>0$ be in (\ref{fep0}) with the replacement of $\hat{v}_*$.
Let $\overline{\sigma_1}, \overline{\sigma_2}, \overline{\sigma_3}$ be as in Lemmas~\ref{fits0}, \ref{bigawayfromPnot0}, \ref{s30} respectively with respect to $\hat{v}_{\sigma_i}$ and $\hat{v}_*$. Set 
\begin{equation}\label{choicesigmanot2} 
    \sigma_0 :=\min\{ \overline{\sigma_1}, \overline{\sigma_2}, \overline{\sigma_3}\}
\end{equation}
satisfying that for $\epsilon >0,$ $\sigma_0$ is sufficiently small such that we have $c_0\left(\frac{\sigma_0}{2^k}\right)^2 \leq \frac{r^2}{2^{2(k+2)}}\frac{16}{3}\epsilon$ for any $k \in \mathbb{N}.$ Define
\begin{align}
    \tilde{v}_k(x)&:=\tilde{v}_{\frac{\sigma_0}{2^k}}(x)=\lambda^v(\frac{\sigma_0}{2^k})v(\frac{\sigma_0}{2^k}x):(B_1(0),g) \to \overline{\bf H}^{k-j},\\    \hat{v}_k(x)&:=\hat{v}_{\frac{\sigma_0}{2^{k}}}(x)=\lambda^{v}(\frac{\sigma_0}{2^k})v^{\eta_0}(\frac{\sigma_0}{2^k}x)=\left(\left({\frac{\sigma_0}{2^k}}\right)^{1-n}\displaystyle{\int_{\partial B_{\frac{\sigma_0}{2^k}}(0)}d^2(v,v(0)) d\Sigma} \right)^{-\frac{1}{2}}v^{\eta_0}(\frac{\sigma_0}{2^k}x),\label{bus2} 
\end{align} 
In particular,  for $k=0$, 
\[
\hat{v}_0(x)=\hat{v}_{\frac{\sigma_0}{2^0}}(x) = \lambda^v(\frac{\sigma_0}{2^0}) v^{\eta_0}(\frac{\sigma_0 x}{2^0})=\lambda^v(\sigma_0)v^{\eta_0}(\sigma_0 x).
\]
We claim that for $k=1, 2, \dots$,
\begin{equation} \label{alt2}
\hat{v}_k(x)= \lambda_{k-1}\hat{v}_{k-1}(\frac{x}{2}), \mbox{ where }\lambda_{k-1}:=\left(2^{n-1}\displaystyle{\int_{\partial B_\frac{1}{2}(0)} d^2(\tilde{v}_{k-1},\tilde{v}_{k-1}(0)) d\Sigma}
\right)^{-\frac{1}{2}}.
\end{equation}
Indeed, assuming (\ref{alt2}) holds for $k=1, \dots, j-1$,  we have
\begin{eqnarray*}
\lambda_{k-1}\hat{v}_{k-1}(\frac{x}{2}) &=& \lambda_{k-1} \hat{v}_{\frac{\sigma_0}{2^{k-1}}}(\frac{x}{2}) = \lambda_{k-1} \lambda^v(\frac{\sigma_0}{2^{k-1}}) v^{\eta_0}(\frac{\sigma_0x}{2^k}),\\
\hat{v}_k(x) & = & \hat{v}_{\frac{\sigma_0}{2^{k}}}(x)=\lambda^v(\frac{\sigma_0}{2^k})v^{\eta_0}(\frac{\sigma_0x}{2^k}),
\end{eqnarray*}
where $\lambda_{k-1}\lambda^{v}(\frac{\sigma_0}{2^{k-1}})=\lambda^v(\frac{\sigma_0}{2^k})$ by an obvious calculation:

\begin{footnotesize}
\begin{eqnarray*}
    \lambda_{k-1}\lambda^{v}(\frac{\sigma_0}{2^{k-1}}) & = & \left(2^{n-1} \int_{\partial B_{\frac{1}{2}}(0)} d^2(\lambda^v(\frac{\sigma_0}{2^{k-1}})v(\frac{\sigma_0}{2^{k-1}}x),v(0)) \, d\Sigma \right)^{-\frac{1}{2}} \cdot \left( \left( \frac{\sigma_0}{2^{k-1}}\right)^{1-n} \int_{\partial B_{\frac{\sigma_0}{2^{k-1}}}(0)} d^2(v,v(0)) \, d\Sigma\right)^{-\frac{1}{2}} \\
    & = & (2^{n-1})^{-\frac{1}{2}} \left( \int_{\partial B_{\frac{1}{2}}(0)} d^2(v(\frac{\sigma_0}{2^{k-1}}x),v(0)) \, d\Sigma \right)^{-\frac{1}{2}} \cdot 
    \left(\left(\frac{\sigma_0}{2^{k-1}}\right)^{1-n} \left( \int_{\partial B_{\frac{\sigma_0}{2^{k-1}}}(0)} d^2(v, v(0)) \, d\Sigma \right) \right)^{\frac{1}{2}} \\ 
    & \cdot & \left( \left( \frac{\sigma_0}{2^{k-1}}\right)^{1-n} \int_{\partial B_{\frac{\sigma_0}{2^{k-1}}}} d^2(v,v(0)) \, d\Sigma\right)^{-\frac{1}{2}}  \\
    & = & (2^{n-1})^{-\frac{1}{2}} \left( \int_{\partial B_{\frac{1}{2}}(0)} d^2(v(\frac{\sigma_0}{2^{k-1}}x),v(0)) \, d\Sigma \right)^{-\frac{1}{2}}\\
    & = & \left( \frac{1}{2}\right)^{\frac{n-1}{2}} \cdot \left( \frac{\sigma_0}{2^{k-1}}\right)^{\frac{n-1}{2}} \left( \int_{\partial B_{\frac{\sigma_0}{2^k}}(0)} d^2(v,v(0)) \, d\Sigma \right)^{-\frac{1}{2}}\\
    & =& \left( \frac{\sigma_0}{2^k} \right)^{\frac{n-1}{2}}\left( \int_{\partial B_{\frac{\sigma_0}{2^k}}(0)} d^2(v,v(0)) \, d\Sigma \right)^{-\frac{1}{2}}\\
    & = & \left( \left( \frac{\sigma_0}{2^k}\right)^{1-n}\int_{\partial B_{\frac{\sigma_0}{2^k}}(0)} d^2(v,v(0)) \, d\Sigma\right)^{-\frac{1}{2}}\\
    & = &\lambda^v(\frac{\sigma_0}{2^k}).
\end{eqnarray*}
\end{footnotesize}

\begin{itemize}
\item Let $\mathcal H=\{\tilde{v}_k: B_1(0) \to \overline{\bf H}^{k-j}\}$ be the sequence of non-constant asymptotically harmonic maps defined above. Define the constant $c$ fixed throughout: 
\begin{equation}\label{cc}
c:=\sup_{\tilde{v}_k \in \mathcal H} \,\left[ \left(2^{n-1}\int_{\partial B_\frac{1}{2} (0)} d^2(\tilde{v}_k,\tilde{v}_k(0)) d\Sigma \right)^{-\frac{1}{2}} \right].
\end{equation}
\item Fix $\epsilon>0$ throughout such that
\begin{equation} \label{choice2}
\frac{16}{3} \epsilon  < \frac{1}{2^2c^2}
\end{equation}
where $c$ is the constant defined above.
\end{itemize}

\begin{remark}
The constant $c$ is bounded away from zero from (\ref{monotonicityformula}) and (\ref{lowerbdlambda2}): For any $k \in \mathbb{Z},$
\[
\left( 2^{n-1} \int_{\partial B_{\frac{1}{2}}(0)} d^2(\tilde{v}_k,\tilde{v}_k(0)) \, d\Sigma\right)^{-\frac{1}{2}}=\lambda^{\tilde{v}_k}(\frac{1}{2})>1.
\]
\end{remark}

\begin{proof}[Proof of Theorem \ref{nosing} for Case II]
We assume that $v^{\eta_0}(x_0)=P_0$ from decomposition (\ref{decomp}).  Let $\{\hat{v}_k\}$ be the blow-up maps at $x_0=0$ defined in (\ref{bus2}). Let $\epsilon, \sigma_0, R, r$ be as in (\ref{choice2}), (\ref{choicesigmanot2}), (\ref{ep0}), (\ref{fep0}) respectively. Observe that $\lambda_{k-1}=\lambda^{\tilde{v}_{k-1}}(\frac{1}{2})$ and $c$ are defined as in (\ref{alt2}) and (\ref{cc}). Recall that $\lambda^{\tilde{v}_k}(\frac{1}{2}) >1$ from (\ref{lowerbdlambda2}). The remainder of the proof proceeds similarly to that of Case I with $v_k$ replaced by $\hat{v}_{k}.$
\end{proof}

\subsection{Proof of Theorem 1.1}\label{pfmain}
Now we are ready to prove Theorem \ref{mainresult} by applying Theorem \ref{nosing}. For $k$-dimensional $(\overline{\mathcal{T}},d_{wp}),$ let $\mathcal{T}'$ be the highest dimensional stratum of $\overline{\mathcal{T}}$ with $\dim(\mathcal{T}')=j \leq k$ such that $u(\Omega) \cap \mathcal{T}' \neq \emptyset.$ This implies that
\[
\mathcal{A} := \{x \in \Omega: u(x) \in \partial \mathcal{T}'\} \neq \Omega.
\]
Consequently, the arguments of Theorem \ref{nosing} imply that $\mathcal{A}= \emptyset$ and hence $u(\Omega) \subset \mathcal{T}'.$

\section{Appendix I: Blow-up Maps into Model Space}
Define the blow-up map $u_{\sigma}:(B_1(0),g) \to \overline{\bf H}$ centered at singular point $x_0=0$ in the way as section \ref{sec:blowup0}. By construction, 
\[
I^{u_{\sigma}}(1):=\int_{\partial B_1(0)} d^2(u_{\sigma},u_{\sigma}(0)) d\Sigma=1.
\]
Since scaling doesn't change the harmonicity and the order, $u_{\sigma}$ is energy minimizing map for any $\sigma$ and
\[
Ord^{u_\sigma}(x_0)=\alpha ,\ \forall \sigma \in (0, \sigma_0].
\]
Note that the energy of $u_{\sigma}$ is bounded: for $\sigma >0$ sufficiently small,
\begin{eqnarray*}
    E^{u_{\sigma}}(1) &=& \int_{B_1(0)} \frac{\sigma^{n-1}}{I^{u}(\sigma)} |\nabla u(\sigma x)|^2 \sigma^2 d\mu
    \\
    & = & (I^{u}(\sigma))^{-1} \sigma^{n+1} \int_{B_{\sigma}(0)} |\nabla u(x)|^2 \sigma^{-n} d\mu
    \\
    & = & \frac{\sigma E^{u}(\sigma)}{I^{u}(\sigma)} \leq 2\, Ord^u (x_0) = 2 \alpha.
\end{eqnarray*}

In other words, $u_{\sigma}$ has uniformly bounded energy on $B_1(0).$ By \cite[Theorem 2.4.6]{2}, $u_{\sigma}$ is uniformly Lipschitz in any compact subset of $B_1(0).$ By \cite[Theorem 3.7]{3}, there exists an abstract NPC space, which we denote by $\overline{\bf H}_*$, and a subsequence $\{u_{\sigma_i}\}$ converging locally uniformly in the pullback sense to the limit map $u_*: B_1(0) \to (\overline{\bf H}_*,d)$ and $u_*$ is also locally uniformly Lipschitz. By \cite[Theorem 3.11]{3}, the limit map $u_*$ is an energy minimizing map to $\overline{\bf H}_*.$ Furthermore, following the argument of \cite[Proposition 3.3]{4}, $u_*$ is an non-constant homogeneous map of order $\alpha,$ i.e. $u_*$ maps every ray from origin in $B_1(0)$ onto a geodesic in $\overline{\bf H}_*$ such that $d(u_*(tx),u_*(0))=t^{\alpha}d(u_*(x),u_*(0)) ,\ t \geq 0.$

\begin{definition}\label{pw0}
    A map $v:B_1(0) \to X$ into an NPC space is \emph{piecewise a function} if, for any connected component $\Omega_0$ of $\{x \in B_1(0): v(x) \neq v(0)\}$, the pullback distance function of $v|_{\Omega_0}$ is equal to the pullback distance function of the function $f:=d(v,v(0))|_{\Omega_0}:\Omega_0 \to \mathbb{R}_+.$
\end{definition}

\begin{lemma} \label{lem:pwfnc0}
Let ${u_{\sigma_i}}$ and limit map $u_*$ be as above, then $u_*$ is piecewise a function.
\end{lemma}
\begin{proof}
    Since $E^{u_{\sigma_i}}(1)$ is uniformly bounded, \cite[Theorem 2.4.6]{2} implies that, for any $r \in (0,1)$, there exists $ C > 0$ such that for any $i$ and $x \in B_r(0)\ \backslash \{x:u_{\sigma_i}(x)=u_{\sigma_i}(0)\},$
    \[
    |\nabla u_{\sigma_i}^{\rho}|(x) \leq C, \qquad (u_{\sigma_i}^{\rho})^3|\nabla u_{\sigma_i}^{\phi}|(x) \leq C.
    \]
    Let $\Omega_0$ be a connected component of $B_1(0)\ \backslash \{x:u_{\sigma_i}(x)=P_0\}$ and $f:\Omega_0 \to \mathbb{R}_+$ by $f(x)=d(u_*(x),u_*(0)).$ Fix $x_{\Omega_0} \in \Omega_0$ and let $K$ be arbitrary compact subset of $\Omega_0$ such that $x_{\Omega_0} \in K.$ Since $u_{\sigma_i} \to u_*$ in pullback sense, we also have local uniform pullback convergence of $u_{\sigma_i}^{\rho} \to f.$ Thus function $u_{\sigma_i}^{\rho}$ is bounded away from $0$ in $K$ for i large enough. Then, the inequality implies $u_{\sigma_i}^{\phi}$ is uniformly Lipschitz in $K$, therefore there exists subsequence $\{u_{\sigma_i}^{\phi}-u_{\sigma_i}^{\phi}(x_{\Omega_0})\}$ (for simplicity we use same notation ${u_{\sigma_i}}$) converging uniformly in $K$ in pullback sense by Arzela-Ascoli Theorem.
    By taking compact exhaustion of $\Omega_0$ and diagonalization procedure, we have that (by taking subsequence if necessary and keeping using the same notation) $\{u_{\sigma_i}^{\phi}-u_{\sigma_i}^{\phi}(x_{\Omega_0})\}$ converges locally uniformly in pullback sense to some function $g$ in $\Omega_0.$
    Thus, $\{(u_{\sigma_i}^{\rho},u_{\sigma_i}^{\phi}-u_{\sigma_i}^{\phi}(x_{\Omega_0}))\}$ converges locally uniformly in $\Omega_0$ to the pair $(f,g):\Omega_0 \to \overline{{\bf H}}_*.$ This convergence is $C^k$ for any $k$ because $\{(u_{\sigma_i}^{\phi},u_{\sigma_i}^{\phi}-u_{\sigma_i}^{\phi}(x_{\Omega_0}))\}$ is sequence of harmonic maps into a smooth Riemannian manifold $\overline{\bf{H}}_*.$
    Since $u_{\sigma_i}$ is harmonic, Euler-Lagrange equation implies in $\Omega_0,$
    \[
    u_{\sigma_i}^{\rho}\Delta u_{\sigma_i}^{\rho}=3(u_{\sigma_i}^{\rho})^6|\nabla u_{\sigma_i}^{\phi}|^2.
    \]
    As $i \to \infty,$
    \[
    f \Delta f=3f^6|\nabla g|^2.
    \]
    Furthermore, order of homogeneity of $f$ in $\Omega_0$, which is equal to the order of homogeneity of $u_*,$ is equal to $\alpha.$
    Thus, since $\Omega_0$ is an open cone, we can rewrite homogeneous function $f$ in polar coordinates: 
    \[
    f(r,\theta)=r^{\alpha}F(\theta),
    \]
    where $F: \Omega_0 \cap \partial B_1(0) \to \overline{\mathbb{R}}_+$ and $\theta=(\theta^1,...,\theta^{n-1})$ are coordinates of ${\bf{S}}^{n-1}.$
    Substituting them into the equation above,
    \[
    r^{2\alpha-2}(\alpha^2F(\theta)+\Delta_{\theta}F)=3r^{6\alpha}F^5(\theta)|\nabla g|^2
    \]
    Since the degrees of $r$-terms on both sides don't agree, to make this equation hold for all $r>0,$ $|\nabla g|^2 \equiv 0.$ Moreover, since $u_{\sigma_i}^{\phi}-u_{\sigma_i}^{\phi}(x_{\Omega_0}) = 0$ as $x=x_{\Omega_0},$ $g(x_{\Omega_0})=0$ and then $g \equiv 0$ in $\Omega_0.$
    So $(u_{\sigma_i}^{\rho},u_{\sigma_i}^{\phi}-u_{\sigma_i}^{\phi}(x_{\Omega_0}))$ converges locally uniformly to $(f,0)$ in $\Omega_0$ in pullback sense. In particular, by definition, we conclude $u_*:B_1(0) \to \overline{\bf H}_*$ is piecewise a function.
\end{proof}

\begin{lemma}\label{geod0}
    Let $\Omega_0$ be a connected component of  $\{x \in B_1(0): u_*(x) \neq u_*(0)\},$ then
    $u_*|_{\Omega_0}$ maps into a geodesic in $\overline{\bf H}_*.$
\end{lemma}
\begin{proof}
    Let $x_{\Omega_0}$ be a point in  $\overline{\Omega_0}$ such that 
    \begin{eqnarray*}
    d(u_*(x_{\Omega_0}),u_*(0)) 
      = \sup_{x \in \overline{\Omega_0}} d(u_*(x),u_*(0)).
    \end{eqnarray*}
By Lemma~\ref{lem:pwfnc0}, $u_*$ is a piecewise function. Thus, 
$$d(u_*(x_1), u_*(x_2))= |f(x_1)-f(x_2)|, \ \ \forall x_1, x_2 \in \Omega_0$$
 where  $f:=d(u_*,u_*(0))|_{\Omega_0}:\Omega_0 \to [0,\infty).$  Extend $f$ to $\Omega_0 \cup \{x_{\Omega_0}\}$ by setting $f(x_{\Omega_0}) =\sup_{x \in \overline{\Omega_0}} d(u_*(x),u_*(0))$.

 Let $x_0 \in \Omega_0$.  Since $f(x_0) =d(u_*(x_0),u_*(0)) \leq \sup_{x \in \overline{\Omega_0}} d(u_*(x),u_*(0))=f(x_{\Omega_0})$, we have
\begin{eqnarray*}
    d(u_*(x_0),u_*(0)) & = &f(x_0)
    \\
    d(u_*(x_{\Omega_0}),u_*(x_0)) & = & | f(x_{\Omega_0})-f(x_0)| = f(x_{\Omega_0})-f(x_0)
    \end{eqnarray*}
Thus,
\begin{eqnarray*}
    d(u_*(x_{\Omega_0}), u_*(0))  & = &      |f(x_{\Omega_0}) - f(0)|
    \\
    & = &  f(x_{\Omega_0})\\
    & = & (f(x_{\Omega_0}) - f(x_0))+f(x_0)\\ 
    & = & d(u_*(x_{\Omega_0}),u_*(x_0)) + d(u_*(x_0),u_*(0))
    \end{eqnarray*}
    which implies that $u_*(x_0)$ is a point on a geodesic from $u_*(0)$ to $u_*(x_{\Omega_0})$ in $\overline{\bf H}_*.$

\end{proof}

\begin{lemma}\label{H_A}
    There exists a totally geodesic isometric embedding $\text{Im} \, u_* \hookrightarrow \overline{\bf H}_A$ 
\end{lemma}
\begin{proof}
Recall (by taking subsequence if necessary) $u_{\sigma_i}|_{\Omega_0}$ converges locally uniformly to $d(u_*,u_*(0))$ in pullback sense. Enumerate the connected components of $B_1(0)\setminus u_*^{-1}(u_*(0))$ by $\Omega_1,...,\Omega_k$ and denote $\Tilde{A}= \{1,...,k\}.$ Claim that there are finitely many connected components. On the contrary, if we have infinitely many $\Omega_1, \Omega_2,...$ in the unit ball $B_1(0),$ as a result, there exists $\Omega_i$ such that $D:= {\bf S}^{n-1} \cap \Omega_i$ has the inradius small sufficiently to zero. The Faber-Krahn inequality implies that the first eigenvalue $\lambda_1$ of $D$ is tending to infinity, which contradicts the equation $\lambda_1(D)=\alpha(\alpha+n-2)$ given in the argument of \cite[Theorem 5.5]{4} where the order $\alpha:=Ord^u(x_0)$ is fixed.

We define an equivalence relation on set $\Tilde{A}$ by $\nu \sim \nu'$ if for any pair of points $z \in \Omega_{\nu}$ and $w \in \Omega_{\nu'},$
\begin{equation}\label{equi}
d(u_*(z),u_*(w)) <  d(u_*(z),u_*(0))+d(u_*(w),u_*(0)).
\end{equation}
To show this is indeed an equivalence relation, let's verify the transitivity property as symmetry and reflexivity are straightforward. Assume $\alpha \sim \beta, \, \beta \sim \eta, \text{ i.e. } \forall x \in \Omega_{\alpha}, \, y \in \Omega_{\beta}, \, z \in \Omega_{\eta},$
\begin{eqnarray*}
d(u_*(x),u_*(y)) & < & d(u_*(x),u_*(0))+d(u_*(y),u_*(0)),
\\
d(u_*(y),u_*(z)) & < & d(u_*(y),u_*(0))+d(u_*(z),u_*(0)).
\end{eqnarray*} 
Note if $|u_{\sigma_i}^{\phi}(x)-u_{\sigma_i}^{\phi}(y)|$ is unbounded when $i$ tending to infinity, then
\begin{eqnarray}\label{sumd0}
d(u_*(x),u_*(y))& = &\lim_{i \to \infty} d_{\overline{\bf H}}(u_{\sigma_i}(x),u_{\sigma_i}(y)) \nonumber\\
& = &\lim_{i \to \infty} \left( d_{\overline{\bf H}}(u_{\sigma_i}(x),u_{\sigma_i}(0)) + d_{\overline{\bf H}}(u_{\sigma_i}(y),u_{\sigma_i}(0))\right) \nonumber\\
& = & d(u_*(x),u_*(0))+d(u_*(y),u_*(0)),
\end{eqnarray} which contradicts the inequalities of equivalence relation.
So we have an upper bound $M$ such that for all $i \in \mathbb{N},$ 
\[
|u_{\sigma_i}^{\phi}(x)-u_{\sigma_i}^{\phi}(y)|<M, \ \ |u_{\sigma_i}^{\phi}(y)-u_{\sigma_i}^{\phi}(z)|<M.
\]
By triangle inequality, for any $i \in \mathbb{N},$
\[
|u_{\sigma_i}^{\phi}(x)-u_{\sigma_i}^{\phi}(z)|\leq|u_{\sigma_i}^{\phi}(x)-u_{\sigma_i}^{\phi}(y)| + |u_{\sigma_i}^{\phi}(y)-u_{\sigma_i}^{\phi}(z)| < 2M.
\]
This boundedness implies that 
\begin{eqnarray*}
    d(u_*(x),u_*(z)) &=& \lim_{i \to \infty} d_{\overline{\bf H}}(u_{\sigma_i}(x),u_{\sigma_i}(z))\\
    &<& \lim_{i \to \infty} \left(d_{\overline{\bf H}}(u_{\sigma_i}(x),u_{\sigma_i}(0))+d_{\overline{\bf H}}(u_{\sigma_i}(z),u_{\sigma_i}(0)) \right)\\
    &=& d(u_*(x),u_*(0))+d(u_*(z),u_*(0)),
\end{eqnarray*}i.e. $\alpha \sim \eta.$

Now we embed the image of $u_*$ into the metric space $\overline{\bf H}_A,$ which is defined in (\ref{H_A0}). Denote the equivalence class containing $\nu \in \tilde A$ by $[\nu]$ and let
$A$ denote the set of equivalence classes of $\Tilde{A}$. Consider $\Omega_{\nu}$ and $\Omega_{\nu'}$ where $\nu \sim \nu'.$ Following the argument in Lemma \ref{lem:pwfnc0}, we choose the representative $\nu$ in $[\nu]$ and define $i_{\nu}:\Omega_{\nu} \to \overline{\bf H}_{[\nu]},$ where $\overline{\bf H}_{[\nu]}$ is a single copy of model space $\overline{\bf H},$ by 
\[i_{\nu}(x)=\Big(d(u_*(x),u_*(0)), \, \lim_{i \to \infty}u_{\sigma_i}^{\phi}(x)-u_{\sigma_i}^{\phi}(x_{\Omega_{\nu}}) \Big).\]
Since $\big| u_{\sigma_i}^{\phi}(x)-u_{\sigma_i}^{\phi}(y)\big|$ is bounded as $i$ tends to $+\infty$ for any $x \in \Omega_{\nu} \text{ and } y \in \Omega_{\nu'},$ $u_{\sigma_i}(y)$ converges as $i$ increases in $\overline{\bf H}$ for each point $y \in \Omega_{\nu'}.$ Therefore, $i_{\nu}$ also maps $\Omega_{\nu'}$ to the same model space $\overline{\bf H}_{[\nu]}.$ For $\nu \nsim \nu',$ there are two induced $i_{\nu} \text{ and } i_{\nu'}$ mapping $\cup_{\nu \in [\nu]} \Omega_{\nu}$ and $\cup_{\nu' \in [\nu']} \Omega_{\nu'}$ into two different model spaces $\overline{\bf H}_{[\nu]} \text{ and } \overline{\bf H}_{[\nu']},$ which is consistent with the metric $d_A$ defined in $\overline{\bf H}_A.$ Combining together, we have the canonical embedding from the image of $u_*$ to $\overline{\bf H}_A.$

\end{proof}

\begin{lemma}\label{|A|}
    $|A| \geq 2.$
\end{lemma}
\begin{proof}

Assume by contradiction that $|A|=1$ i.e. $u_*$ maps all connected components into one model space $\overline{\bf H}$ such that $\big| u_{\sigma_i}^{\phi}(x)-u_{\sigma_i}^{\phi}(y) \big|$ is bounded for any $x,y$ 
in $B_1(0).$ Since $u_*(B_1(0))$ is the set of geodesic segments, define $\gamma(t)$ as the geodesic extension of $u_*(\Omega_k)$ and fix a point $\gamma(t_0)$ sufficiently far from $P_0$ such that $d_{\overline{\bf H}}(\gamma(t_0),u_*(\tilde{x})) < d_{\overline{\bf H}}(\gamma(t_0),u_*(0))$ for any $\tilde{x} \in \partial B_1(0).$ Consider the subharmonic function $d_{\overline{\bf H}}(\gamma(t_0),u_*(x))$ defined on $B_1(0).$ This function achieves its maximum at $0 \in B_1(0),$ which contradicts the maximum principle for subharmonic functions. 

\end{proof}

\section{Appendix II: Non-constant Pullback Limit $v_*$}

Let $u_{\sigma}: (B_1(0),g) \to \mathcal{U} \times \mathcal{V}$ be the blow-up maps at the singular point $x_0=0 \in \mathcal{S}_j(u)$ defined in section \ref{sec:blowup}. By the computation in Appendix I, $I^{u_{\sigma}}(1)=1$ and $E^{u_{\sigma}}(1)$ is bounded. As $\sigma \to 0,$ we have the {\it sequence of blow-up maps} at $x_0:$
\[
\{u_{\sigma_i}=(V_{\sigma_i},v_{\sigma_i})=(V_{\sigma_i},v^{1}_{\sigma_i},...,v^{k-j}_{\sigma_i}): (B_1(0),g) \to \mathbb{C}^j \times \overline{\bf H}_1 \times ... \times \overline{\bf H}_{k-j}\}
\]
where $\overline{\bf H}_{\eta}$ is one single copy of $\overline{\bf H} \text{ for } \eta=1,...,k-j.$ By \cite[Lemma 4.19]{1}, $\{v_{\sigma_i}: B_1(0) \to \overline{\bf H}^{k-j}\}$ is a sequence of asymptotically harmonic maps with $v_{\sigma_i}(0)=P_0.$ In particular, $\{v^{\eta}_{\sigma_i}: B_1(0) \to \overline{\bf H}\}$ is a sequence of asymptotically harmonic maps with $v^{\eta}_{\sigma_i}(0)=P_0$ for each $\eta=1,...,k-j.$ Then, \cite[Lemma 4.10]{1} implies that there exists subsequence 
\[
    v_{\sigma_i} \to v_*=(v_*^{1},...,v_*^{k-j}): (B_1(0),g) \to (\overline{\bf H}_*^{k-j}=\overline{\bf H}_* \times ... \times \overline{\bf H}_*,d) 
\]locally uniformly in pullback sense, where $v_*$ is a homogeneous harmonic map to a product of NPC spaces. In section \ref{sec:blowup1} we assume that the component $v_*^{\eta_0}$ is non-constant for some $\eta_0 \in \{1,...,k-j\}.$ Denote $v_{\sigma_i}^{\eta_0} \text{ and } v_*^{\eta_0}$ by $v_{\sigma_i} \text{ and } v_*$ for simplicity. To overcome the difficulty that $v_{\sigma_i}$ is non-harmonic, we introduce the approximating harmonic map $w_i:$

\begin{lemma}\label{re}
Let $\{v_{\sigma_i}\}$ be the blow-up sequence and non-constant limit $v_*$ be as above, then there exists a sequence $\{w_i\}$ of approximating harmonic maps such that in any compact subset $K$ of $B_1(0),$
\begin{equation}\label{approx}
   \lim_{i \to \infty} \sup_K d_{\overline{\bf H}}(v_{\sigma_i},w_i)=0. 
\end{equation}
\end{lemma}
\begin{proof}
Recall that $v_{\sigma_i} \to v_*$ locally uniformly in pullback sense. Let $K \subset \subset B_1(0) \text{ and } w_i:K \to \overline{\bf H}$ be the harmonic map such that $w_i|_{\partial K}=v_{\sigma_i}|_{\partial K}.$ Without loss of generality, let
\[
K=B_{\frac{3}{4}}(0), \ R=\vartheta=\frac{3}{4}, \ r=\frac{2}{3}.
\]
By Definition \ref{asy}(ii), there exists constant $E_0 >0$ such that 
\[
E^{w_i}\left(\frac{3}{4}\right) \leq E^{v_{\sigma_i}}\left(\frac{3}{4}\right) \leq \left(\frac{3}{4}\right)^nE_0 < \infty.
\]
Then, for a fixed $z_0 \in \partial B_{\frac{3}{4}}(0)$ and any $x \in B_{\frac{3}{4}}(0),$
\begin{center}
    $d_{\overline{\bf H}}(w_i(x),w_i(z_0))$ is uniformly bounded on $B_{\frac{3}{4}}(0).$
\end{center}
Combined with Definition \ref{asy}(iii), for $i$ large sufficiently, for any $x \in B_{\frac{3}{4}}(0),$
\begin{eqnarray*}
    d_{\overline{\bf H}}(w_i(x),P_0) &\leq& d_{\overline{\bf H}}(w_i(x),w_i(z_0))+d_{\overline{\bf H}}(w_i(z_0),P_0)\\
    &=& d_{\overline{\bf H}}(w_i(x),w_i(z_0))+d_{\overline{\bf H}}(v_{\sigma_i}(z_0),v_{\sigma_i}(0))\\
    &\leq&c < \infty.
\end{eqnarray*}
Thus, Definition \ref{asy}(iv) implies
    
\[
\lim_{i \to \infty} \, \sup_{B_{\frac{1}{2}}(0)}d_{\overline{\bf H}}(v_{\sigma_i},w_i) =0,
\]
i.e. $\sup d_{\overline{\bf H}}(v_{\sigma_i},w_i) \to 0$ holds in any compact subset of $B_1(0).$
\end{proof}

\begin{lemma} \label{lem:pwfnc}
Let ${v_{\sigma_i}}$ and non-constant limit map $v_*$ be as above, then $v_*$ is piecewise a function.
\end{lemma}
\begin{proof}
    Let $B_r(0) \text{ where } r \in (0,1).$ From Lemma \ref{re}, we have the sequence $\{v_{\sigma_i} |_{B_r(0)} \to \overline{\bf H}\}$ and the sequence $\{w_i|_{B_r(0)} \to \overline{\bf H}\}$ of approximating harmonic maps with \[
    \sup_{B_r(0)} d^2_{\overline{\bf H}}(v_{\sigma_i},w_i) \to 0.\]
    Lemma \ref{lem:pwfnc0} implies that
    $(w_i^{\rho},w_i^{\phi}-w_i^{\phi}(x_{\Omega_0}))$ converges locally uniformly in pullback sense to $(f,0)=d(w_*,w_*(0))$ in connected component $\Omega_0.$ Lemma \ref{re} and \cite[Lemma 4.10]{1} implies that in $\Omega_0 \cap B_r(0),$    
    \begin{center}
        $(v_{\sigma_i}^{\rho},v_{\sigma_i}^{\phi}-v^{\phi}_{\sigma_i}(x_{\Omega_0})) \to (d(v_*,v_*(0)),0)$ locally uniformly in pullback sense. 
    \end{center}
    By Definition \ref{pw0}, we conclude $v_*:B_1(0) \to \overline{\bf H}_*$ is piecewise a function.
\end{proof}

\begin{remark}
    Since the harmonic homogeneous pullback limit $v_*$ is piecewise a function, Lemmas \ref{geod0}, \ref{H_A} and \ref{|A|} still hold for $v_*.$
\end{remark}

\section{Appendix III: Constant Pullback Limit $v_*$}
Let $\tilde{v}_{\sigma}: (B_1(0),g) \to \overline{\bf H}^{k-j}$ be the alternative blow-up map at $x_0=0$ defined in section \ref{sec:blowup2}. By construction, 
\[
I^{\tilde{v}_{\sigma}}(1):=\int_{\partial B_1(0)} d^2(\tilde{v}_{\sigma},\tilde{v}_{\sigma}(0)) d\Sigma=1.
\]
Notice that the energy of $\tilde{v}_{\sigma}$ is bounded: for $\sigma >0$ sufficiently small,
\begin{eqnarray*}
    E^{\tilde{v}_{\sigma}}(1) &=& \int_{B_1(0)} \frac{\sigma^{n-1}}{I^{v}(\sigma)} |\nabla v(\sigma x)|^2 \sigma^2 d\mu
    \\
    & = & (I^{v}(\sigma))^{-1} \sigma^{n+1} \int_{B_{\sigma}(0)} |\nabla v(x)|^2 \sigma^{-n} d\mu
    \\
    & = & \frac{\sigma E^{v}(\sigma)}{I^{v}(\sigma)} \leq 2\, Ord^v (x_0) = 2 \beta.
\end{eqnarray*}
In section \ref{sec:blowup2}, we have the subsequence of blow-up component maps $\{\hat{v}_{\sigma_i}=\tilde{v}_{\sigma_i}^{\eta_0}:(B_1(0),g) \to \overline{\bf H} \}$ converging to the non-constant homogeneous harmonic limit map $\hat{v}_*=\tilde{v}^{\eta_0}_*:(B_1(0),g) \to (\overline{\bf H}_*,d)$ locally uniformly in pullback sense in that
\[
d_{\overline{\bf H}}(\hat{v}_{\sigma_i}(\cdot),\hat{v}_{\sigma_i}(\cdot)) \to d(\hat{v}_*(\cdot),\hat{v}_*({\cdot})) \text{ in compact subsets of $B_1(0)$.}
\]

\begin{lemma}\label{re2}
Let $\{\hat{v}_{\sigma_i}\}$ be the blow-up sequence and non-constant limit $\hat{v}_*$ be as above, then there exists a sequence $\{\hat{w}_i\}$ of approximating harmonic maps such that in any compact subset $K$ of $B_1(0),$
\begin{equation}\label{approx2}
   \lim_{i \to \infty} \sup_K d_{\overline{\bf H}}(\hat{v}_{\sigma_i},\hat{w}_i)=0. 
\end{equation}
\end{lemma}
\begin{proof}
Same as the proof in Lemma \ref{re}.
\end{proof}

\begin{lemma} \label{lem:pwfnc2}
Let ${\hat{v}_{\sigma_i}}$ and non-constant limit map $\hat{v}_*$ be as above, then $\hat{v}_*$ is piecewise a function.
\end{lemma}
\begin{proof}
Same as the proof in Lemma \ref{lem:pwfnc} by replacing $v_{\sigma_i}, v_*,w_i,$ and $w_*$ with $\hat{v}_{\sigma_i}, \hat{v}_*,\hat{w}_i,$ and  $\hat{w}_*$.
\end{proof}

\begin{remark}
    The fact that non-constant homogeneous harmonic pullback limit $\hat{v}_*$ is piecewise a function implies that Lemmas \ref{geod0}, \ref{H_A} and \ref{|A|} still hold with $u_*, u_{\sigma_i}$ replaced by $\hat{v}_*$ and $\hat{v}_{\sigma_i}.$
\end{remark}

\end{document}